\documentclass[10pt,oneside,reqno,a4paper]{amsart}
\RequirePackage{fix-cm} % arbitrary font scaling
\usepackage[T1]{fontenc}

\usepackage[english]{babel}
\usepackage{amssymb,amsmath,amsthm,graphicx,amsfonts}
\usepackage[dvipsnames]{xcolor}

\definecolor{mygray}{RGB}{0,56,105}
\usepackage{cite,enumerate,setspace}

\usepackage{hyperref}
\hypersetup{
 colorlinks,
 citecolor=mygray,
 linkcolor=mygray,
 urlcolor=mygray}
 
\usepackage{esint}
\usepackage{geometry}

\usepackage{caption}
\makeatletter
\def\@seccntformat#1{\hspace*{0mm}%
 \protect\textup{\protect\@secnumfont
   \ifnum\pdfstrcmp{subsection}{#1}=0 \bfseries\fi% subsection # in \bfseries
   \csname the#1\endcsname
   \protect\@secnumpunct
     }%
}

\numberwithin{equation}{section}

\newcommand{\assign}{:=}
\newcommand{\mathd}{\mathrm{d}}
\newcommand{\of}{:}
\newcommand{\tmabbr}[1]{#1}

\newcommand{\tmmathbf}[1]{\ensuremath{\boldsymbol{#1}}}

\newcommand{\tmname}[1]{\textsc{#1}}
\newcommand{\tmop}[1]{\ensuremath{\operatorname{#1}}}

\newcommand{\tmsamp}[1]{\textsf{#1}}
\newcommand{\tmstrong}[1]{\textbf{#1}}
\newcommand{\tmtextit}[1]{\emph{#1}}
\newcommand{\tmtextsc}[1]{\text{{\scshape{#1}}}}
\newcommand{\tmtextup}[1]{\text{{\upshape{#1}}}}

\newtheorem{lemma}{Lemma}
\newtheorem{proposition}{Proposition}
{\theoremstyle{remark}\newtheorem{remark}{Remark}}
\newtheorem{theorem}{Theorem}
\newcommand{\jdeg}{\mathrm{deg\,}}
\newcommand{\m}{\tmmathbf{m}}
\newcommand{\n}{\tmmathbf{n}}
\newcommand{\uu}{\tmmathbf{u}}
\newcommand{\vv}{\tmmathbf{v}}
\newcommand{\e}{\tmmathbf{e}}
\newcommand{\GG}{\mathcal{G}}
\newcommand{\MM}{\mathcal{M}}

\newcommand{\T}{\top}
\newcommand{\cpt}{\tmtextup{)}}
\newcommand{\opt}{\tmtextup{(}}

\newcommand{\Stwo}{\mathbb{S}}
\newcommand{\RR}{\mathbb{R}}
\newcommand{\ZZ}{\mathbb{Z}}
\newcommand{\NN}{\mathbb{N}}
\newcommand{\CC}{\mathbb{C}}

\newcommand{\grad}{\nabla}

\begin{document}

\title[On symmetry of energy minimizing harmonic-type maps]{On symmetry of energy minimizing harmonic-type maps\\
on cylindrical surfaces}

\author{Giovanni Di Fratta}

\author{Alberto Fiorenza}

\author{Valeriy Slastikov}

\begin{abstract}
  The paper concerns the analysis of global minimizers of a Dirichlet-type energy functional in the class of $\Stwo^2$-valued maps defined in cylindrical surfaces. The model naturally arises as a curved thin-film limit in the theories of nematic liquid crystals and micromagnetics. We show that minimal configurations are $z$-invariant and that energy minimizers in the class of weakly axially symmetric competitors are, in fact, axially symmetric. Our main result is a family of \emph{sharp} Poincar{\'e}-type inequality on the circular cylinder, which allows for establishing a nearly complete picture of the energy landscape. The presence of symmetry-breaking phenomena is highlighted and discussed. Finally, we provide a complete characterization of in-plane minimizers, which typically appear in numerical simulations for reasons we explain.
  
  \medskip\noindent\textbf{Keywords.} Poincaré inequality, harmonic maps, magnetic skyrmions
  
  \noindent\textbf{AMS subject classifications.} 35A23, 35R45, 49R05, 49S05, 82D40
\end{abstract}

{\maketitle}
\thispagestyle{empty}
\section{Introduction}
{\noindent}The interplay between geometry and topology plays a fundamental
role in many fields of applied science. The most basic examples include thin
nematic liquid crystal shells~{\cite{Napoli,Miller_2013,Serra_2016}} and
curvilinear magnetic
nanostructures~{\cite{gaididei2014curvature,streubel2016magnetism}}. Curvature
effects and topological constraints lead to unusual properties of the
underlying physical systems and promote the appearance of novel
microstructures, providing a promising way to design new materials with
prescribed properties.

In the last decade, magnetic systems with curvilinear shapes have been subject
to extensive experimental and theoretical research
({\tmabbr{cf.}}~{\cite{carbou2001thin,Slastikov2012,slastikov2005micromagnetics,Di_Fratta_2019,streubel2016magnetism,Di_Fratta_2020,Ignat_2021}}).
Recent advances in the fabrication of magnetic spherical hollow nanoparticles
and rolled-up nanomembranes with a cylindrical shape lead to the creation of
artificial materials with unexpected characteristics and numerous applications
in nanotechnology, including high-density data storage, magnetic logic, and
sensor devices
({\tmabbr{cf.}}~{\cite{hu2008core,Schmidt_2001,streubel2016magnetism}}).
Embedding planar structures in the three-dimensional space permits altering
their magnetic properties by tailoring their local curvature. The interplay
between geometry, topology, and Dzyaloshinskii--Moriya interaction (DMI) leads
to the formation of novel magnetic spin textures, e.g., chiral domain walls
and skyrmions~{\cite{BabaA2022, Davoli2020, Davoli2022, di2019weak,Di_Fratta_Mon_2022, muratov2017domain}}.
The curvature effects have been shown to play a crucial role in stabilizing these chiral
spin-textures. Spherical and cylindrical thin films are of particular interest
due to their simple geometry and capability to host {\emph{spontaneous}}
skyrmions (topologically protected magnetic structures) even in the absence of
DMI~{\cite{gaididei2014curvature,kravchuk2012out,streubel2016magnetism}}.

In what follows, occasionally, we are going to use the language of
micromagnetics. However, our mathematical results apply to
other physical systems (e.g., the Oseen-Frank theory of nematic liquid crystals).

\subsection{State of the art}It is well established that, when the thickness
of a thin shell is very small relative to the lateral size of the system, the
demagnetizing field interactions behave, at the leading order, as a local
shape-anisotropy, see
{\cite{gioia1997micromagnetics,carbou2001thin,di2019variational,Di_Fratta_2020}}.
In the context of a thin curvilinear shell (generated by extruding a regular
surface $\mathcal{M}$ in $\RR^3$ along the normal direction), the
leading-order contribution to the micromagnetic energy functional reads as
{\cite{Di_Fratta_2020,di2019variational}}:

\begin{equation}
  \mathcal{E}: \m \in H^1( \MM, \Stwo^2)
  \mapsto \int_{\MM} \left| \grad_{\xi} \tmmathbf{m} (\xi) \right|^2
  \mathd \xi + \alpha \int_{\MM} (\m (\xi) \cdot
  \n (\xi))^2 \mathd \xi, \label{eq:sphericshelllimitgen}
\end{equation}
where $\tmmathbf{n}$ is the normal field to the surface $\MM$, $\alpha
\in \RR$ is an effective anisotropy parameter accounting for both
{\emph{shape}} and {\emph{crystalline}} anisotropy, and $\grad_{\xi}$ is the
tangential gradient on $\MM$.

The role of $\alpha$ is easy to understand qualitatively. Uniform states are
the only local minimizers of $\mathcal{E}$ when $\alpha = 0$. For
$\tmop{large} \alpha > 0$, tangential vector fields are energetically favored;
for $\tmop{large} \alpha < 0$, i.e., when shape anisotropy prevails over
perpendicular crystal anisotropy, energy minimization prefers normal vector
fields.

An exact characterization of the minimizers of this problem is a
nontrivial task with far-reaching consequences for modern magnetic storage
technologies~{\cite{kravchuk2016topologically}}. Recently, a partial answer
about the structure of minimizers of $\mathcal{E}$ has been given for the case
$\MM= \Stwo^2$. It has been shown that (see {\cite{Di_Fratta_2019}})
\begin{enumerate}
 \item for any $\alpha \in \RR$, the normal vector fields $\pm
  \n$ are stationary points of the energy functional $\mathcal{E}$
  on the space $H^1 ( \Stwo^2, \Stwo^2)$; moreover, they are
  strict local minimizers for every $\alpha < 0$ and are unstable for $\alpha
  > 0$.
  
 \item When $\alpha \leqslant - 4$, the normal vector fields $\pm
  \n$ are the {\emph{only}} {\emph{global}} minimizers of
  $\mathcal{E}$.
\end{enumerate}

Also, in~{\cite{melcher2019curvature}} it is shown that for $\alpha \ll - 1$,
skyrmionic solutions topologically distinct from the ground state emerge as
excited states.

For $\alpha > 0$, the energy landscape of $\mathcal{E}$ is challenging to
describe. Indeed, topological obstructions ({\emph{hairy ball theorem}})
prevent the existence of purely tangential vector fields in $H^1(
\Stwo^2, \Stwo^2)$. Numerical simulations suggest that when $\alpha >
0$, the energy $\mathcal{E}$ can exhibit magnetic states with skyrmion number
$0$ or $\pm 1$ (see,
e.g.,~{\cite{kravchuk2012out,kravchuk2016topologically,sloika2017geometry}}).
Also, within the homotopy class $\left\{ \jdeg \m= 0 \right\}$, the
energy $\mathcal{E}$ favors the so-called {\emph{onion}} state if $\alpha$ is
sufficiently small, and the {\emph{vortex}} state otherwise.

Classifying the ground states in spherical thin films in the regime $\alpha>0$ is demanding. However, many of the difficulties one faces and the emerging symmetry-breaking phenomena are already present in the analysis of ground states and axially symmetric solutions in the more tractable geometry of a cylinder. This observation triggered our interest in the questions addressed in this paper and led to developing some techniques that we believe can be further improved to tackle ground states' analysis in more complex geometries.

\subsection{Contributions of present work}{\noindent}Let $\Gamma \subseteq
\RR^2$ be the image of a smooth Jordan curve $\zeta : [0, 2 \pi] \rightarrow
\Gamma$, and let $\mathcal{C} \assign I \times \Gamma$, $I \assign [- 1, 1]$,
be the cylindrical surfaces generated by $\Gamma$ (see
Figure~\ref{fig:Cylinder}). Given that $\m$ is $\Stwo^2$-valued, we
have that up to the constant term $- \alpha | \mathcal{C} |$, with $|
\mathcal{C} |$ being the area of $\mathcal{C}$, the minimization problem for
\eqref{eq:sphericshelllimitgen} is equivalent to the minimization of the
energy functional
\begin{equation}
  \m \in H^1( \mathcal{C}, \Stwo^2) \mapsto
  \int_{\mathcal{C}} \left| \grad_{\xi} \m (\xi) \right|^2 \mathd
  \xi - \alpha \int_{\mathcal{C}} | \m (\xi) \times \n
  (\xi) |^2 \mathd \xi . \label{eq:tempen2}
\end{equation}
The previous expression \eqref{eq:tempen2} is more convenient for the
following reason. When $\alpha = 0$, any constant $\Stwo^2$-valued vector
field is a minimizer with zero minimal energy. The scenario is still trivial
when $\alpha > 0$. There are only two minimizers in this regime, and these are
the constant vector fields $\pm \tmmathbf{e}_3 = \pm (0, 0, 1)$ whose
associated minimal energy is $- \alpha | \mathcal{C} |$. However, the
situation suddenly becomes engaging when $\alpha < 0$. This is the regime this
paper is devoted to, and working with \eqref{eq:tempen2} allows dealing with
nonnegative energies whereas \eqref{eq:sphericshelllimitgen} does not.
Therefore, we set $\alpha \assign - \kappa^2$, with $\kappa^2 \neq 0$, and,
from now on, we focus our investigations on the energy functional
\begin{equation}
  \mathcal{E} (\m) \assign \int_{\mathcal{C}} \left| \grad_{\xi}
  \m \right|^2 \mathd \xi + \kappa^2 \int_{\mathcal{C}} |
  \m \times \n |^2 \mathd \xi, \qquad \m \in H^1( \mathcal{C}, \Stwo^2) . \label{eq:micromagenfunonC}
\end{equation}
Here $\grad_{\xi}$ stands for the tangential gradient on $\mathcal{C}$, and
$\kappa^2$ is a positive constant that controls the perpendicular anisotropy's
strength. Note that, equivalently, the value of $\kappa^2$ controls the size
of the sample $\mathcal{C}$. Indeed, simple rescaling allows reducing the
analysis of \eqref{eq:micromagenfunonC} to a scaled cylinder and a different
value of $\kappa^2$. This is why when we later analyze the minimizers
$\mathcal{E}$ on {\emph{circular}} cylinders, we focus only on $\mathcal{C}
\assign I \times \Stwo^1$.
\begin{figure}[t]
  {\includegraphics[width=\textwidth]{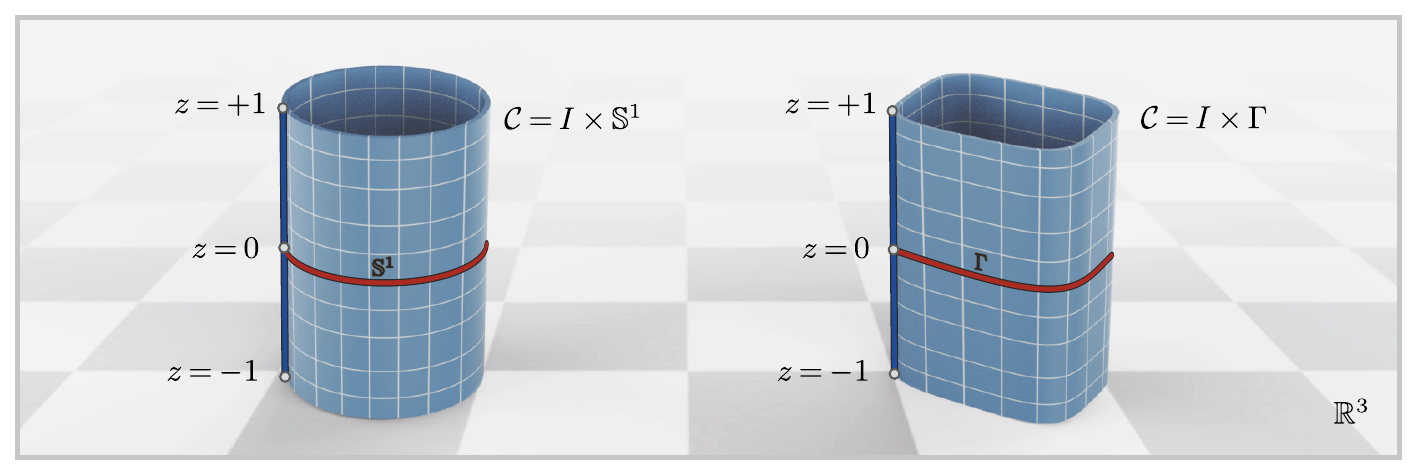}}
  \caption{\label{fig:Cylinder}The paper analyzes ground states of the energy
  functional $\mathcal{E}$ in the admissible class of $\Stwo^2$-valued maps
  defined on cylindrical surfaces $\mathcal{C}= I \times \Gamma$. After a
  general result on the $z$-invariance of the minimizers of $\mathcal{E}$ in
  $H^1( \mathcal{C}, \Stwo^2)$, we look for the analytic
  expression of the minimizers in circular cylinders (i.e., in the case
  $\Gamma = \Stwo^1$).
  }
\end{figure}

This paper's main aim concerns the analysis of global minimizers
of the energy functional \eqref{eq:micromagenfunonC} in the class of
$\Stwo^2$-valued maps defined in the {\emph{circular}} cylinder $\mathcal{C}=
I \times \Stwo^1$. The analysis we perform involves several steps.

First, for any $\kappa^2 > 0$, we show that minimizers of the energy
$\mathcal{E}$ defined in \eqref{eq:micromagenfunonC} are $z$-invariant. In
Proposition~\ref{Prop:dimred} we prove the result holds under the more general
framework of cylindrical surfaces of the type $\mathcal{C} \assign I \times
\Gamma$ where $I \assign [- 1, 1]$ and $\Gamma \subseteq \RR^2$ is the image
of a smooth Jordan curve $\zeta : [0, 2 \pi] \rightarrow \Gamma$ (see
Figure~\ref{fig:Cylinder}). Also, we prove that when $\mathcal{C}= I \times
\Stwo^1$, $z$-invariance of the minimizers holds in the restricted class of
{\emph{weakly}} axially symmetric configurations which are defined by the
condition that
\begin{equation}
  \int_{\Stwo^1} \m_{\bot} (z, \gamma) \mathd \gamma = 0 \quad
  \forall z \in I, \label{eq:was}
\end{equation}
where $\m_{\bot} \assign \m - (\m \cdot \e_3) \e_3$. It
is simple to show that every axially symmetric configuration satisfies
\eqref{eq:was} ({\tmabbr{cf.}}~Remark~\ref{eq:rmkaxsymm}). In
Theorem~\ref{cor:axsymmmin}, we prove that every minimizer of $\mathcal{E}$ in
the class of {\emph{weakly}} axially symmetric competitors is, in fact,
axially symmetric. The proof is based on a symmetrization argument in
conjunction with the classical Poincar{\'e}-Wirtinger inequality for null
average and periodic functions.

Second, we focus on the analysis of global minimizers of the energy
$\mathcal{E}$ in the unrestricted class $H^1( I \times \Stwo^1, \Stwo^2)$, i.e., when no {\emph{weak}} axial symmetry is assumed on the
competitors. Our main result is a family of sharp Poincar{\'e}-type
inequalities (see Theorem~\ref{thm:PoincIneq}), which allows us to establish
the following picture of the energy landscape of $\mathcal{E}$ (see
Theorem~\ref{thm:BFTCylinder}).
\begin{enumerate}
  \item If $\kappa^2 \geqslant 3$, the normal vector fields $\pm
  \n$ are the only global minimizers of the energy functional
  $\mathcal{E}$ in $H^1( \mathcal{C}, \Stwo^2)$.
  
  \item Moreover, they are strict local minimizers for every
  $\kappa^2 > 1$ and unstable for $0 < \kappa^2 < 1$. The constant vector
  fields $\pm \e_3$ are unstable for all $\kappa^2 > 0$.
\end{enumerate}
The sharp Poincar{\'e}-type inequality is stated in
Theorem~\ref{thm:PoincIneq} and states that for every $\kappa^2 > 0$ there
holds
\begin{equation}
  \int_{\Stwo^1} \left| \grad_{\gamma} \uu \right|^2 \mathd \gamma +
  \kappa^2 \int_{\Stwo^1} | \tmmathbf{u} \times \n |^2 \mathd \gamma
  \; \geqslant \; c^2_{\kappa} \int_{\Stwo^1} | \uu |^2 \mathd \gamma
  \quad \forall \uu \in H^1( \Stwo^1, \RR^3) .
  \label{eq:PIpreview}
\end{equation}
with $c^2_{\kappa} = 1$ if $\kappa^2 \geqslant 3$, $c^2_{\kappa} = \frac{1}{2}
(\kappa^2 - \omega_{\kappa}^2 + 4)$ if $0 < \kappa^2 \leqslant 3$, and
$\omega_{\kappa}^2 \assign \sqrt{\kappa^4 + 16}$. Our result also includes a
precise characterization of the minimizers for which the equality sign is
reached in \eqref{eq:PIpreview}. For the proof, we work in Fourier space.
While the frequencies decouple nicely, the vector-valued nature of $H^1(
\Stwo^1, \RR^3)$ elements, as well as the presence of the anisotropic
constant $\kappa^2$, has the effect that different space directions strongly
interact, and this requires careful analysis (see {\cite{Di_Fratta_2019}}).

Third, motivated by their importance in numerical simulations (see
Section~\ref{sec:numericalinplane} for a detailed discussion), we investigate
global minimizers of $\mathcal{E}$ in the class of in-plane configurations. We
show that if $\m_{\bot} \in H^1 (\Stwo^1, \Stwo^1)$ is the
{\emph{profile}} of a minimizer in $H^1 (\mathcal{C}, \Stwo^1)$ of the energy
functional $\mathcal{E}$, then either $\jdeg  \m_{\bot} = 0$ or $\jdeg 
\m_{\bot} = 1$. Indeed, among other things, in Theorem~\ref{thm:mainS1S1} we
show the existence of a threshold value $\kappa^2_{\ast}$ of the anisotropy
parameter for which the following energetic implications hold:
\begin{enumerate}
  \item If $\kappa^2 > \kappa_{\ast}^2$, then any global minimizer
  has degree one and, moreover, for every $\kappa^2 > 0$, the normal fields
  $\pm \n$ are the only two minimizers in the homotopy class
  $\left\{ \jdeg  \m_{\bot} = 1 \right\}$.
  
  \item If $\kappa^2 < \kappa_{\ast}^2$, then any global minimizer
  has degree zero, and an accurate analytic description is given in terms of
  elliptic integrals.
\end{enumerate}
The previous two points allow for a complete characterization of the energy
landscape of in-plane minimizers. The normal vector fields $\pm \n$
are the only two in-plane energy minimizers when $\kappa^2 > \kappa_{\ast}^2$
and the common minimum value of the energy is $2 \pi$. Instead, when $\kappa^2
< \kappa_{\ast}^2$, the minimal energy depends on $\kappa^2$, and the precise
minimal values, as well as the analytic expressions of the minimizers, are
given in terms of elliptic integrals (see
\eqref{eq:minmdegzero}-\eqref{eq:minendegzero}).

\subsection{Outline}The paper is organized as follows. In
Section~\ref{sec:ASM} we prove that minimal configurations are $z$-invariant
({\tmabbr{cf.}}~Proposition~\ref{Prop:dimred}) and that every minimizer of
$\mathcal{E}$ in the class of {\emph{weakly}} axially symmetric competitors
is, in fact, axially symmetric ({\tmabbr{cf.}}~Theorem~\ref{cor:axsymmmin}).
Section~\ref{se:SPI} is devoted to the analysis of global minimizers of the
energy $\mathcal{E}$ in the unrestricted class $H^1( I \times \Stwo^1,
\Stwo^2)$. Our main result is a family of sharp Poincar{\'e}-type
inequalities ({\tmabbr{cf.}}~Theorem~\ref{thm:PoincIneq}), which allows
for establishing a nearly complete picture of the energy landscape of
$\mathcal{E}$ ({\tmabbr{cf.}}~Theorem~\ref{thm:BFTCylinder}). Finally, in
Section~\ref{sec:numericalinplane}, we provide a complete characterization of
the energy landscape of in-plane minimizers of $\mathcal{E}$.

\subsection{Notation}In what follows, for a given embedded submanifold
$\MM$ of $\RR^3$, we indicate by $H^1( \MM, \RR^m)$, $m \geqslant 1$, the Sobolev space of vector-valued functions
defined on $\MM$, endowed with the norm
({\tmabbr{cf.}}~{\cite{Agranovich_2015}})
\begin{equation}
  \| \uu \|^2_{H^1( \MM, \RR^m)} \assign
  \int_{\MM} | \uu (\xi) |^2 \mathd \xi + \int_{\MM}
  \left| \grad_{\xi} \uu (\xi) \right|^2 \mathd \xi \,,
\end{equation}
where $\grad_{\xi} \uu$ is the tangential gradient of $\uu$
on $\MM$. We write $H^1( \MM, \Stwo^2)$ for the
metric subspace of $H^1( \MM, \RR^3 )$ consisting of
vector-valued functions with values in the unit $2$-sphere of $\RR^3$. When
$\MM \assign \Stwo^1$ is the unit $1$-sphere, $H^1( \Stwo^1,
\RR^m )$ identifies to the Sobolev space $H^1_{\sharp}( [- \pi,
\pi], \RR^m)$ consisting of $2 \pi$-periodic vector-valued functions
and endowed with the norm
\begin{equation}
  \| \uu \|^2_{H^1_{\sharp} ( [- \pi, \pi], \RR^m )}
  \assign \int_{- \pi}^{\pi} | \uu (t) |^2 \mathd t + \int_{-
  \pi}^{\pi} | \partial_t \uu (t) |^2 \mathd t.
\end{equation}
Finally, we denote by $H^1_{\sharp} ( [- \pi, \pi], \Stwo^1)$ and
$H^1_{\sharp} ( [- \pi, \pi], \Stwo^2)$ the metric subspaces of
the Sobolev space $H^1_{\sharp} ( [- \pi, \pi], \RR^m )$ consisting, respectively, of
$\Stwo^1$-valued and $\Stwo^2$-valued periodic functions.

\section{Symmetry properties of the minimizers}\label{sec:ASM}

{\noindent}Our first result, stated in the next Proposition~\ref{Prop:dimred}, shows
that for any $\kappa^2 > 0$ every minimizer $\m (z, t)$ of the energy
$\mathcal{E}$ in \eqref{eq:micromagenfunonC} is $z$-invariant. We state the
result in the more general framework of cylindrical surfaces of the type
$\mathcal{C} \assign I \times \Gamma$ where $I \assign [- 1, 1]$ and $\Gamma
\subseteq \RR^2$ is the image of a smooth Jordan curve $\zeta : [0, 2 \pi]
\rightarrow \Gamma$. Note that, by parameterizing the cylinder $\mathcal{C}$
through the map
\begin{equation}
  \gamma (z, t) \assign (z,\zeta (t)),
\end{equation}
we can rewrite the energy functional \eqref{eq:micromagenfunonC} in the form
\begin{equation}
  \mathcal{E} (\m) = \int_{- 1}^1 \int_{\Gamma} \left| \grad_{\zeta}
  \m (z, \zeta) \right|^2 + | \partial_z \m (z, \zeta) |^2
  \mathd \zeta \mathd z + \kappa^2 \int_{- 1}^1 \int_{\Gamma} \, |
  \m (z, \zeta) \times \n (\zeta) |^2 \mathd \zeta \mathd
  z. \label{eq:micromagenfunonCccoordgen}
\end{equation}
This equivalent expression of the energy in \eqref{eq:micromagenfunonC} is used in the proof of the following result on the $z$-invariance of energy minimizers.
\begin{proposition}[\tmname{$z$-invariance of energy minimizers}]
  \label{Prop:dimred}Let $\m \in H^1( \mathcal{C}, \Stwo^2
 )$ be a {\opt}global{\cpt} minimizer of the micromagnetic energy
  functional {\emph{\eqref{eq:micromagenfunonC}}} with $\mathcal{C} \assign I
  \times \Gamma$ and $\Gamma \subseteq \RR^2$ the image of a smooth Jordan
  curve $\zeta : [0, 2 \pi] \rightarrow \Gamma$. Then there exists a minimizer
  $\m_{\ast} \in H^1( \mathcal{C}, \Stwo^2)$ of $\mathcal{E}$
  in {\emph{\eqref{eq:micromagenfunonCccoordgen}}}, built from $\m$, which is
  $z$-invariant, i.e., such that
  \begin{equation}
    \m_{\ast} (z, \zeta) =\uu_{\ast} (\zeta)
    \label{eq:zinvm}
  \end{equation}
  for some $\uu_{\ast} \in H^1( \Gamma, \Stwo^2)$.
  Actually, every minimizer of $\mathcal{E}$ has the form
  {\emph{\eqref{eq:zinvm}}} for some minimizer $\uu_{\ast} \in H^1( \Gamma, \Stwo^2)$ of the reduced energy
  \begin{equation}
    \mathcal{F} (\uu) \assign \int_{\Gamma} \left| \grad_{\zeta}
    \uu (\zeta) \right|^2 \mathd \zeta \; + \; \kappa^2 \int_{\Gamma}
    | \uu (\zeta) \times \n (\zeta) |^2 \mathd \zeta .
    \label{eq:reden}
  \end{equation}
\end{proposition}

\begin{remark}
  In general, $z$-invariance does not hold for critical points of the energy.
  In fact, when $\mathcal{C} \assign I \times \Stwo^1$ and $\kappa^2 = 1$, the
  helices satisfy the Euler--Lagrange equations associated with $\mathcal{E}$
  and are not $z$-invariant ({\tmabbr{cf.}}~Proposition
  \ref{prop:nullavperp}). The observation implies that the helical
  configurations predicted in {\cite{streubel2016magnetism}} are critical
  points of the energy, but not ground states.
\end{remark}

\begin{proof}
  We use a symmetrization argument. Let $\m$ be a minimizer of
  $\mathcal{E}$, and let us consider the function (note that, $\n
  (z, \zeta) =\n (\zeta)$ is $z$-invariant)
  \begin{equation}
    \Phi : z \in I \assign [- 1, 1] \mapsto \int_{\Gamma} \left| \grad_{\zeta}
    \m (z, \zeta) \right|^2 \mathd \zeta + \kappa^2 \int_{\Gamma} |
    \m (z, \zeta) \times \n (\zeta) |^2 \mathd \zeta .
    \label{eq:defPhi}
  \end{equation}
  In terms of $\Phi$ the energy functional \eqref{eq:micromagenfunonC} reads
  as
  \begin{equation}
    \mathcal{E} (\m) = \int_{- 1}^1 \left( \Phi (z) + \int_{\Gamma}
    | \partial_z \m (z, \zeta) |^2 \mathd \zeta \right) \mathd z.
    \label{eq:toregremark}
  \end{equation}
  Note that since $\m$ minimizes $\mathcal{E}$ on the two-dimensional surface
$\mathcal{C}$, $\m$ is smooth in $\mathcal{C}$. Indeed, the Euler-Lagrange
equations for $\m$ fit into the class of almost harmonic maps treated in
{\cite[Chapter~4]{Moser2005}}. In particular,
({\tmabbr{cf.}}~{\cite[Theorem~4.2]{Moser2005}}), $\m$ is H{\"o}lder
continuous and, therefore, by the usual bootstrap argument, smooth in the
interior of $\mathcal{C}$. Moreover, through a classical reflection argument (across
any of the two boundary components in $\partial \mathcal{C}= \partial I \times
\Gamma$) one can easily show that minimizers are smooth up to the boundary.
In particular, $\Phi$ is continuous on $[- 1, 1]$
  and $\tmop{argmin}_{z \in [- 1, 1]} \Phi (z) \neq \emptyset$. We arbitrarily
  choose a point $z_{\ast} \in \tmop{argmin}_{z \in [- 1, 1]} \Phi (z)$ and,
  with that, we define the $z$-invariant configuration
  \begin{equation}
    \m_{\ast} (z, \zeta) \assign \m (z_{\ast}, \zeta)
    \quad \text{for every } (z, \zeta) \in I \times \Gamma .
    \label{eq:defmstar}
  \end{equation}
  Note that, since $\m$ is smooth in $\mathcal{C}$,
  $\m_{\ast}$ is well-defined. Taking into account that for every
  $\xi \assign (z, \zeta) \in I \times \Gamma$ we have
  \begin{equation}
    \left| \grad_{\xi} \m (z, \zeta) \right|^2 = \left|
    \grad_{\zeta} \m (z, \zeta) \right|^2 + | \partial_z
    \m (z, \zeta) |^2
  \end{equation}
  with $\grad_{\zeta} \m$ the tangential gradient on $\zeta$, from
  \eqref{eq:micromagenfunonCccoordgen} and \eqref{eq:defPhi} we get that
  \begin{align}
    \mathcal{E} (\m_{\ast}) =\; & \int_{- 1}^1 \int_{\Gamma} \left|
    \grad_{\zeta} \m_{\ast} (z, \zeta) \right|^2 \mathd \zeta \mathd
    z + \kappa^2 \int_{- 1}^1 \int_{\Gamma} | \m_{\ast} (z, \zeta)
    \times \n (\zeta) |^2 \mathd \zeta \mathd z \nonumber\\
    =\; & \int_{- 1}^1 \int_{\Gamma} \left| \grad_{\zeta} \m
    (z_{\ast}, \zeta) \right|^2 \mathd \zeta \mathd z + \kappa^2 \int_{- 1}^1
    \int_{\Gamma} | \m (z_{\ast}, \zeta) \times \n (\zeta)
    |^2 \mathd \zeta \mathd z \nonumber\\
    =\; & \int_{- 1}^1 \Phi (z_{\ast}) \mathd z \nonumber\\
     \leqslant\; & \int_{- 1}^1 \Phi (z) \mathd z + \int_{- 1}^1
    \int_{\Gamma} | \partial_z \m (z, \zeta) |^2 \mathd \zeta \mathd
    z \nonumber\\
    =\; & \mathcal{E} (\m) .  \label{eq:chaininequalities}
  \end{align}
  Hence, if $\m$ is a minimizer in $H^1( \mathcal{C}, \Stwo^2
 )$ of \eqref{eq:micromagenfunonC} then so is the $z$-invariant
  configuration $\m_{\ast} (z, \zeta) \assign \m
  (z_{\ast}, \zeta)$. Moreover, if $\m$ is any minimizer in $H^1( \mathcal{C}, \Stwo^2)$, then, with $\m_{\ast}$
  defined as in \eqref{eq:defmstar}, we get that $\mathcal{E}
  (\m_{\ast}) =\mathcal{E} (\m)$. This entails that all
  the inequalities in \eqref{eq:chaininequalities} are, in fact, equalities.
  Therefore,
  \begin{equation}
    \int_{- 1}^1 \int_{\Gamma} | \partial_z \m (z, \zeta) |^2 \mathd
    \zeta \mathd z = 0,
  \end{equation}
  from which we conclude that $\m$ is $z$-invariant. This completes
  the proof.
\end{proof}

Since we are interested in the symmetry-breaking phenomena of minimizers, we want
to focus on the symmetric case when $\Gamma$ is unit circle $\Stwo^1$.
Parameterizing the cylinder $\mathcal{C} \assign I \times \Stwo^1$ through the
map
\begin{equation}
  \gamma (z, t) \assign (\cos t) \tmmathbf{e}_1 + (\sin t) \tmmathbf{e}_2 +
  z\tmmathbf{e}_3, \quad t \in [- \pi, \pi]
\end{equation}
with $\tmmathbf{e}_1, \tmmathbf{e}_2, \tmmathbf{e}_3$ the standard basis of
$\RR^3$, we can rewrite \eqref{eq:micromagenfunonCccoordgen} in the following
form
\begin{equation}
  \mathcal{E} (\m) = \int_{- 1}^1 \int_{- \pi}^{\pi} | \partial_t
  \m (z, t) |^2 + | \partial_z \m (z, t) |^2 \mathd t
  \mathd z + \kappa^2 \int_{- 1}^1 \int_{- \pi}^{\pi} | \m (z, t)
  \times \n (t) |^2 \mathd z \mathd t
  \label{eq:micromagenfunonCccoord}
\end{equation}
where we made the common abuse of notation
\begin{equation}
  \m (z, t) \assign (\m \circ \gamma) (z, t), \qquad
  \n (t) \assign (\n \circ \gamma) (z, t) = (\cos t, \sin
  t, 0) .
\end{equation}
According to Proposition \ref{Prop:dimred}, the energy landscape associated
with \eqref{eq:micromagenfunonC} is completely characterized as soon as one
describes the minimizers in $H^1( \Stwo^1, \Stwo^2)$ of the
energy functional ({\tmabbr{cf.}} \eqref{eq:reden})
\begin{equation}
  \mathcal{F} (\uu) \assign \int_{\Stwo^1} \left| \grad_{\gamma}
  \uu (\gamma) \right|^2 \mathd \gamma \; + \; \kappa^2
  \int_{\Stwo^1} | \uu (\gamma) \times \n (\gamma) |^2
  \mathd \gamma . \label{eq:redenmain}
\end{equation}
Note that, in terms of the standard (conformal) parameterization of $\Stwo^1$
given by $\gamma : t \in [- \pi, \pi] \mapsto (\cos t) \tmmathbf{e}_1 + (\sin
t) \tmmathbf{e}_2$, the energy functional $\mathcal{F}$ reads as
\begin{equation}
  \mathcal{F} (\uu) = \int_{- \pi}^{\pi} | \partial_t
  \uu (t) |^2 \mathd t + \kappa^2 \int_{- \pi}^{\pi} | \uu
  (t) \times \n (t) |^2 \mathd t  \label{eq:1denS1}
\end{equation}
with, again, the convenient abuse of notation
\begin{equation}
  \uu (t) \assign (\uu \circ \gamma) (t), \qquad
  \n (t) \assign (\n \circ \gamma) (t) = (\cos t, \sin t,
  0) . \label{eq:localpar}
\end{equation}
For the next result, stated in Proposition \ref{prop:nullavperp}, we introduce
further notation. For each $\uu \in H^1( \Stwo^1, \Stwo^2
)$ we denote by $\uu_{\bot}$ the projection of $\uu$
on $\RR^2 \times \{ 0 \}$, namely, $\uu_{\bot} \assign (u_1, u_2, 0)$
if $\uu= (u_1, u_2, u_3)$. Also, we denote by
\begin{equation}
  \langle \uu_{\bot} \rangle \assign \frac{1}{2 \pi} \int_{-
  \pi}^{\pi} \uu_{\bot} (t) \mathd t
\end{equation}
the average of $\uu_{\bot}$ on $\Stwo^1$ and, for any $\theta \in [-
\pi, \pi]$, we set
\begin{equation}
  \uu_{\theta} (t) \assign R (t) \left(\begin{array}{c}
    \sin \theta\\
    0\\
    \cos \theta
  \end{array}\right) = (\sin \theta) \n+ (\cos \theta)
  \tmmathbf{e}_3, \qquad R (t) \assign \left(\begin{array}{ccc}
    \cos t & - \sin t & 0\\
    \sin t & \cos t & 0\\
    0 & 0 & 1
  \end{array}\right) . \label{eq:utheta}
\end{equation}
For every $t \in [- \pi, \pi]$ the action of $R (t)$ is a rotation through an
angle $t$ about the $z$-axis.

In order to prove Proposition~\ref{prop:nullavperp} we need the sharp form of
the Poincar{\'e}-Wirtinger inequality on $\Stwo^1$ that we recall here; its
proof is a trivial application of Parseval's theorem for Fourier series and is
therefore omitted.

\begin{proposition}[\tmname{Poincar{\'e}-Wirtinger inequality}]~\label{prop:SPI}If $\uu \in
  H^1_{\sharp} ( [- \pi, \pi], \RR^2)$ is null-average {\opt}i.e.,
  $\langle \uu_{\bot} \rangle = 0${\cpt} then
  \begin{equation}
    \int_{- \pi}^{\pi} | \uu (t) |^2 \leqslant \int_{- \pi}^{\pi} |
    \partial_t \uu (t) |^2 . \label{eq:SPI}
  \end{equation}
  The minimizer is reached when $\uu (t) =\tmmathbf{c}_1 \cos t
  +\tmmathbf{c}_2 \sin t$, for arbitrary constant vectors $\tmmathbf{c}_1,
  \tmmathbf{c}_2 \in \RR^2$.
\end{proposition}

We can now state Proposition \ref{prop:nullavperp} and Theorem
\ref{cor:axsymmmin}, which are our main results about axially symmetric
minimizers. Their proof is given at the end of this section.

\begin{proposition}[\tmname{axially symmetric energy minimizers}]
  \label{prop:nullavperp}In the class of configurations $\uu \in H^1( \Stwo^1, \Stwo^2)$ such that $\langle \uu_{\bot}
  \rangle = 0$, the only global minimizers of {\emph{\eqref{eq:redenmain}}}
  are given by
  \begin{equation}
    \left\{ \begin{array}{llll}
      \uu & = & \pm \tmmathbf{e}_3 & \text{if } 0 < \kappa^2 <
      1,\\
      \uu & = & \uu_{\theta} \quad & \text{if } \kappa^2
      = 1,\\
      \uu & = & \pm \n \quad & \text{if } \kappa^2 > 1,
    \end{array} \right. \label{eq:charnullavm}
  \end{equation}
  with $\uu_{\theta}$ given by {\emph{\eqref{eq:utheta}}}. Thus, if
  $0 < \kappa^2 < 1$ or $\kappa^2 > 1$, there exist only two minimizers, while
  when $\kappa^2 = 1$ there exist infinitely many minimizers described by
  $\uu_{\theta}$ with $\theta$ chosen arbitrarily in $[- \pi, \pi]$.
  The corresponding values of the minimal energies are given by
  \begin{equation}
    \left\{ \begin{array}{ll}
      2 \pi \kappa^2 & \text{if } 0 < \kappa^2 \leqslant 1,\\
      2 \pi & \text{if } \kappa^2 \geqslant 1.
    \end{array} \right.
  \end{equation}
\end{proposition}

\begin{remark}
  Note that at $\kappa^2 = 1$ a symmetry-breaking phenomenon appears. The
  minimizers suddenly pass from the in-plane configurations $\pm \n$
  for $\kappa^2 > 1$, to the purely axial configurations $\pm \tmmathbf{e}_3$
  for $\kappa^2 < 1$. Also, note that if $\tmmathbf{e} \in \Stwo^1 \times \{ 0
  \}$ is in-plane, then $\mathcal{F} (\tmmathbf{e}) = \pi \kappa^2$.
  Therefore, for $0 < \kappa^2 < 1$, the configurations $\pm \tmmathbf{e}_3$
  are never global minimizers outside of the restricted admissible class of
  {\emph{weakly}} axially symmetric configurations (i.e., maps $\uu
  \in H^1( \Stwo^1, \Stwo^2)$ such that $\langle
  \uu_{\bot} \rangle = 0$). A similar observation applies to the
  configurations $\pm \n$ when $0 < \kappa^2 < 2$ (because of
  $\mathcal{F} (\tmmathbf{e}) = \pi \kappa^2$); in this range of parameters
  $\pm \n$ cannot be global minimizers unless we restrict the
  minimization problem to the class of axially symmetric configurations.
\end{remark}

Before stating our main result on axially symmetric minimizers, we give the
following definition. We say that $\m \in H^1( \mathcal{C},
\Stwo^2)$ is {\emph{weakly}} axially symmetric (with respect to the
$z$-axis) if
\begin{equation}
  \langle \m_{\bot} (z, \cdot) \rangle_{\Stwo^1} \assign \frac{1}{2
  \pi} \int_{\Stwo^1} \m_{\bot} (z, \gamma) \mathd \gamma = 0 \quad
  \forall z \in I. \label{eq:condprojnull}
\end{equation}
\begin{remark}
  \label{eq:rmkaxsymm}It is important to stress that every axially symmetric
  configuration satisfies \eqref{eq:condprojnull}. Indeed, if $\m$
  is axially symmetric with respect to the $z$-axis then, in local
  coordinates, i.e., through the parameterization of $\Stwo^1$ given by
  $\gamma : t \in [- \pi, \pi] \mapsto (\cos t) \tmmathbf{e}_1 + (\sin t)
  \tmmathbf{e}_2$, we have that
  \begin{equation}
    \m (z, t) = R (t) \tilde{\m} (z) \quad \forall (z, t)
    \in I \times [- \pi, \pi]
  \end{equation}
  for some profile $\tilde{\m} \in H^1( I, \Stwo^2)$.
  Hence, $\langle \m (z, \cdot) \rangle_{\Stwo^1} =
  (\tilde{\m} (z) \cdot \tmmathbf{e}_3) \tmmathbf{e}_3$ for every $z
  \in I$, and this implies that $\langle \m_{\bot} (z, \cdot)
  \rangle_{\Stwo^1} = 0$ for every $z \in I$.
  Also, note that the class of {\emph{weakly}} axially symmetric
  configurations it is not directly related to the class of null-average
  configurations in $H^1( \mathcal{C}, \Stwo^2)$. Even if $\m$ is
  $z$-invariant, \eqref{eq:condprojnull} does not imply that $\m$ is null
  average, but only that its projection $\m_{\bot}$ is null average.
\end{remark}

\begin{theorem}[\tmname{axially symmetric energy minimizers}]
  \label{cor:axsymmmin}Let $\mathcal{C} \assign I \times \Stwo^1$, with $I =
  [- 1, 1]$. Assume that $\m$ is a {\opt}global{\cpt} minimizer of
  the micromagnetic energy functional {\emph{\eqref{eq:micromagenfunonC}}} in
  the class of weakly axially symmetric configurations. Then, $\m$
  is $z$-invariant and, more precisely, the following assertions hold:
  \begin{enumerate}
    \item[\emph{i.}] If $0 < \kappa^2 < 1$ then necessarily $\m \in \{ \pm
    \tmmathbf{e}_3 \}$.
    
    \item[\emph{ii.}] If $\kappa^2 > 1$ then necessarily $\m \in \{ \pm
    \n \}$.
    \item[\emph{iii.}] When $\kappa^2 = 1$, there are infinitely many axially symmetric
    minimizers; they are all $z$-invariant and given by $\m (z, t)
    =\uu_{\theta} (t)$ with $\theta \in [- \pi, \pi]$ and
    $\uu_{\theta}$ given by {\emph{\eqref{eq:utheta}}}.
  \end{enumerate}
  The values of the minimal energies are given by $4 \pi \kappa^2$ if $0 <
  \kappa^2 \leqslant 1$ and by $4 \pi$ if $\kappa^2 \geqslant 1$.
\end{theorem}

\begin{remark}
  Note that, due to Remark \ref{eq:rmkaxsymm} and the fact that $\pm \n$ and
  $\pm \e_3$ are axially symmetric {\opt}with respect to the $z$-axis{\cpt},
  the conclusions of Theorem~\ref{cor:axsymmmin} still hold in the class of
  axially symmetric minimizers.
\end{remark}

\begin{remark}
  We stress that Theorem~\ref{cor:axsymmmin} does not look for axially
  symmetric minimizers in the class of minimizers of $\mathcal{E}$. In other
  words, axially symmetric minimizers do not need to be global minimizers. In
  fact, Theorem~\ref{cor:axsymmmin} characterizes the minimizers of
  $\mathcal{E}$ in the class of configurations that satisfy condition
  \eqref{eq:condprojnull} and shows that, in this class, the minimizers are
  necessarily $z$-invariant and axially symmetric.
\end{remark}

We first give the proof of Proposition~\ref{prop:nullavperp}, then we prove
Theorem~\ref{cor:axsymmmin} as a consequence of Proposition~\ref{Prop:dimred},
Proposition~\ref{prop:nullavperp}, and Remark~\ref{eq:rmkaxsymm}.

\begin{proof}{Proof of Proposition~\ref{prop:nullavperp}.}
  For every $t \in [- \pi, \pi]$ we denote by $R (t)$ the rotation through an
  angle $t$ about the $z$-axis which appears
  ({\tmabbr{cf.}}~\eqref{eq:utheta}). Clearly, $\n (t) = R (t)
  \tmmathbf{e}_1$. Next, let $\uu \in H^1( \Stwo^1, \Stwo^2
 )$ be a minimizer of \eqref{eq:1denS1} and let us choose $t_{\ast} \in
  [- \pi, \pi]$ such that
  \begin{equation}
    t_{\ast} \in \underset{t \in [- \pi, \pi]}{\arg \: \min}  \left( |
    \uu_{\bot} (t) |^2  + \kappa^2  | \uu (t) \times
    \n (t) |^2 \right)
  \end{equation}
  with $\uu_{\bot} (t) \assign (u_1 (t), u_2 (t), 0)$. Here, we are identifying $\uu$ with its H{\"o}lder continuous representative
so that $t_{\ast}$ is well-defined. Define the new configuration
  \begin{equation}
    \uu^{\ast} (t) \assign R (t) R^{\T} (t_{\ast}) \uu
    (t_{\ast}) .
  \end{equation}
  We then have $| \uu^{\ast} (t) | = 1$ and $\langle
  \uu^{\ast}_{\bot} \rangle = 0$ because $\left\langle R (t) R^{\T}
  (t_{\ast}) \uu (t_{\ast}) \right\rangle = (\uu (t_{\ast})
  \cdot \tmmathbf{e}_3) \tmmathbf{e}_3$. Moreover,
  \begin{equation}
    | \partial_t \uu^{\ast} (t) |^2 \; = \; \left| R^{\T}
    \partial_t R (t) R^{\T} (t_{\ast}) \uu (t_{\ast}) \right|^2 \;
    = \; \left| \tmmathbf{e}_3 \times \left( R^{\T} (t_{\ast}) \uu
    (t_{\ast}) \right) \right|^2 \; = \; | \uu_{\bot} (t_{\ast})
    |^2
  \end{equation}
  and
  \begin{equation}
    | \uu^{\ast} (t) \times \n (t) |^2 \; = \; \left|
    R^{\T} (t_{\ast}) \uu (t_{\ast}) \times \tmmathbf{e}_1 \right|^2
    \; = \; \; | \uu (t_{\ast}) \times \n (t_{\ast}) |^2
  \end{equation}
  It follows that
  \begin{align}
    \mathcal{F} (\uu^{\ast}) =\; & \int_{- \pi}^{\pi} | \partial_t
    \uu^{\ast} (t) |^2 \mathd t + \kappa^2 \int_{- \pi}^{\pi} |
    \uu^{\ast} (t) \times \n (t) |^2 \; \mathd t \\
    =\; & \int_{- \pi}^{\pi} | \uu_{\bot} (t_{\ast}) |^2 +
    \kappa^2 | \uu (t_{\ast}) \times \n (t_{\ast}) |^2  \;
    \mathd t \\
     \leqslant \; & \int_{- \pi}^{\pi} | \uu_{\bot} (t) |^2 + \kappa^2
    | \uu (t) \times \n (t) |^2  \; \mathd t. 
    \label{eq:tempPW1}
  \end{align}
  After that, the sharp Poincar{\'e}-Wirtinger inequality on $\Stwo^1$
  ({\tmabbr{cf.}}~Proposition~\ref{prop:SPI}) assures that for every
  $\uu_{\bot} \in H^1( \Stwo^1, \RR^2)$ such that
  $\langle \uu_{\bot} \rangle =\tmmathbf{0}$ one has
  \begin{equation}
    \int_{- \pi}^{\pi} | \uu_{\bot} (t) |^2 \; \mathd t \leqslant
    \int_{- \pi}^{\pi} | \partial_t \uu_{\bot} (t) |^2 \mathd t.
    \label{eq:tempPW2}
  \end{equation}
  Combining \eqref{eq:tempPW1} and \eqref{eq:tempPW2} we conclude that
  \begin{equation}
    \mathcal{F} (\uu^{\ast}) \leqslant \int_{- \pi}^{\pi} |
    \partial_t \uu_{\bot} (t) |^2 + \kappa^2 | \uu (t)
    \times \n (t) |^2  \; \mathd t \; \leqslant \; \mathcal{F}
    (\uu) .
  \end{equation}
  Thus $\uu^{\ast}$ and $\uu$ are both minimizers. This
  implies that
  \begin{equation}
    \mathcal{F} (\uu) = \int_{- \pi}^{\pi} | \partial_t
    \uu_{\bot} (t) |^2 + \kappa^2 | \uu (t) \times
    \n (t) |^2  \; \mathd t = \mathcal{F} (\uu^{\ast}) .
  \end{equation}
  It follows that whenever $\langle \uu_{\bot} \rangle
  =\tmmathbf{0}$, then necessarily $\partial_t (\uu (t) \cdot
  \tmmathbf{e}_3) = 0$ on $\Stwo^1$. On the other hand, the equality
  $\mathcal{F} (\uu^{\ast}) =\mathcal{F} (\uu)$ also entails
  that the equality sign is reached in the sharp Poincar{\'e}-Wirtinger
  inequality \eqref{eq:SPI}, i.e., that
  \begin{equation}
    \int_{- \pi}^{\pi} | \uu_{\bot} (t) |^2 \mathd t = \int_{-
    \pi}^{\pi} | \partial_t \uu_{\bot} (t) |^2 \mathd t
  \end{equation}
  whenever $\uu$ is a minimizer with $\langle \uu_{\bot}
  \rangle =\tmmathbf{0}$. However, the equality sign in the
  Poincar{\'e}-Wirtinger inequality is achieved if, and only if,
  $\uu_{\bot} = (\cos t) \tmmathbf{a}_1 + (\sin t) \tmmathbf{a}_2$
  for some $\tmmathbf{a}_1, \tmmathbf{a}_2 \in \RR^2 \times \{ 0 \}$.
  Combining this observation with the conditions $\partial_t (\uu (t)
  \cdot \tmmathbf{e}_3) = 0$ and $| \uu | = 1$ we conclude that if
  $\langle \uu_{\bot} \rangle =\tmmathbf{0}$ then necessarily
  \begin{equation}
    \uu (t) =\uu_{\theta} (t) \assign \left(\begin{array}{c}
      \sin \theta \cos t\\
      \sin \theta \sin t\\
      \cos \theta
    \end{array}\right)
  \end{equation}
  for some $\theta \in \RR$. Finally, we note that with $\uu (t)$
  given by the previous expression, we have
  \[ | \partial_t \uu_{\theta} (t) |^2 = \sin^2 \theta, \qquad |
     \uu_{\theta} (t) \times \n (t) |^2 = 1 - \sin^2
     \theta . \]
  Therefore
  \begin{equation}
    \mathcal{F} (\uu_{\theta} (t)) = \int_{- \pi}^{\pi} \sin^2 \theta
    + \kappa^2 (1 - \sin^2 \theta) \mathd t \; = \; 2 \pi [\kappa^2 +
    \sin^2 \theta (1 - \kappa^2)] . \label{eq:exprenF}
  \end{equation}
  Minimizing \eqref{eq:exprenF} with respect to $\theta \in [- \pi, \pi]$ we
  get that $\theta = \pm \pi$ when $0 < \kappa^2 < 1$ and, in this case,
  \begin{equation}
    \mathcal{F} (\uu_{\theta} (t)) =\mathcal{F} (\pm \tmmathbf{e}_3)
    = 2 \pi \kappa^2 .
  \end{equation}
  Also, we get that the angle $\theta$ can be arbitrarily chosen when
  $\kappa^2 = 1$, and in this case,
  \begin{equation}
    \mathcal{F} (\uu_{\theta} (t)) = 2 \pi \quad \forall \theta
    \in [- \pi, \pi] .
  \end{equation}
  Finally, we obtain that $\theta = \pm \pi / 2$ when $\kappa^2 > 1$ and, in
  this case,
  \begin{equation}
    \mathcal{F} (\uu_{\theta} (t)) = 2 \pi .
  \end{equation}
  This gives \eqref{eq:charnullavm} and completes the proof.
\end{proof}

\begin{proof}{Proof of Theorem \ref{cor:axsymmmin}}
  The proof is a consequence of Proposition~\ref{Prop:dimred},
Proposition~\ref{prop:nullavperp}, and Remark~\ref{eq:rmkaxsymm}. The only
point is to realize that the proof of Proposition~\ref{Prop:dimred} is not
affected by the introduction of the additional weakly axially symmetric
constraint \eqref{eq:condprojnull}. Indeed, the only place where the
constraint \eqref{eq:condprojnull} impacts the proof of
Proposition~\ref{Prop:dimred} is in the regularity of minimizers which we
based on {\cite[Theorem~4.2]{Moser2005}}, and one can show that the linearity
of the constraint \eqref{eq:condprojnull} makes the arguments in
{\cite[Theorem~4.2]{Moser2005}} still work. However, for the reader's
convenience, we give here an alternative proof of
Proposition~\ref{Prop:dimred} that does not rely on the theory developed in
{\cite[Chapter~4]{Moser2005}} and immediately adapts to the presence of the
additional constraint \eqref{eq:condprojnull}. This will complete the proof of  Theorem~\ref{cor:axsymmmin}.

\medskip\noindent{\emph{Proof of Proposition~{\emph{\ref{Prop:dimred}}} under the additional
constraint {\emph{\eqref{eq:condprojnull}}}}}. We assume that $\Stwo^1 = \zeta ([0, 2 \pi])$ is parameterized by
arc-length. Let $\m_0$ be a minimizer of
\begin{equation}
  \mathcal{E} \left( \m \right) = \int_{- 1}^1 \int_0^{2 \pi} \left|
  \nabla \m (z, t) \right|^2 + \kappa^2 | \m (z, t) \times \n (t) |^2 \mathd z
  \mathd t, \nonumber
\end{equation}
and let $\m_{\varepsilon} \in C^{\infty} \left( I \times \Gamma, \Stwo^2
\right)$ be such that $\m_{\varepsilon} \rightarrow \m_0$ in $H^1 \left( I
\times \Gamma, \Stwo^2 \right)$ (see {\cite[p.267]{Schoen_1983}}). \ For every
$\varepsilon > 0$ we consider the function (which now depends on
$\varepsilon$)
\begin{equation}
  \Phi_{\varepsilon} : z \in I \assign [- 1, 1] \mapsto \int_0^{2 \pi} \left|
  \partial_t \m_{\varepsilon} (z, t) \right|^2 \mathd t + \kappa^2  |
  \m_{\varepsilon} (z, t) \times \n (t) |^2 \mathd t. \label{eq:defPhinew}
\end{equation}
In terms of $\Phi_{\varepsilon}$ the energy reads as
\begin{equation}
  \mathcal{E} (\m_{\varepsilon}) = \int_{- 1}^1 \left( \Phi_{\varepsilon} (z)
  + \int_0^{2 \pi} | \partial_z \m_{\varepsilon} (z, t) |^2 \mathd t \right)
  \mathd z. \label{eq:toregremarkr2new}
\end{equation}
For every $\varepsilon > 0$, $\Phi_{\varepsilon}$ is continuous on $[- 1, 1]$
and $\tmop{argmin}_{z \in [- 1, 1]} \Phi_{\varepsilon} (z) \neq \emptyset$. We
arbitrarily choose a point $z_{\varepsilon} \in \tmop{argmin}_{z \in [- 1, 1]}
\Phi_{\varepsilon} (z)$ and, with that, we define the $z$-invariant
configuration $\tmmathbf{u}_{\varepsilon} (t) \assign \m_{\varepsilon}
(z_{\varepsilon}, t)$. We then have
\begin{align}
  \mathcal{E} (\tmmathbf{u}_{\varepsilon}) = \hspace{0.27em} & \int_{- 1}^1
  \int_0^{2 \pi} \left| \partial_t \m_{\varepsilon} (z_{\varepsilon}, t)
  \right|^2 \mathd t + \kappa^2  | \m_{\varepsilon} (z_{\varepsilon}, t)
  \times \n (t) |^2 \mathd t \mathd z \nonumber\\
  = \hspace{0.27em} & \int_{- 1}^1 \Phi_{\varepsilon} (z_{\varepsilon}) \mathd
  t \leqslant \int_{- 1}^1 \Phi_{\varepsilon} (z) \mathd z \leqslant \int_{-
  1}^1 \left( \Phi_{\varepsilon} (z) + \int_0^{2 \pi} | \partial_z
  \m_{\varepsilon} (z, t) |^2 \mathd t \right) \mathd z \nonumber\\
  = \hspace{0.27em} & \mathcal{E} (\m_{\varepsilon}) . 
  \label{eq:chaininequalitiesr2new}
\end{align}
Since $\mathcal{E} (\m_{\varepsilon})$ is bounded, we have that
$\tmmathbf{u}_{\varepsilon}$ is bounded in $H^1( I_{2 \pi},
\RR^3 )$, $I_{2\pi}:=(0,2\pi)$, and, therefore, there exists $\tmmathbf{u}_0 \in H^1
( I_{2 \pi}, \RR^3 )$ such that $\tmmathbf{u}_{\varepsilon}
\rightharpoonup \tmmathbf{u}_0$ weakly in $H^1 (I_{2 \pi},
\RR^3 )$ and, up to a subsequence, $| \tmmathbf{u}_0 | = 1$ (this is
because of $\tmmathbf{u}_{\varepsilon} (t) \assign \m_{\varepsilon}
(z_{\varepsilon}, t)$ with $\left| \m_{\varepsilon} \right| = 1$). Passing to
the limit, we then have that
\[ \mathcal{E} (\tmmathbf{u}_0) \leqslant \liminf_{\varepsilon \rightarrow 0}
   \mathcal{E} (\tmmathbf{u}_{\varepsilon}) \leqslant \int_{- 1}^1 \Phi_0 (z)
   \mathd z \leqslant \int_{- 1}^1 \left( \Phi_0 (z) + \int_0^{2 \pi} |
   \partial_z \m_0 (z, t) |^2 \mathd t \right) \mathd z =\mathcal{E} (\m_0) .
\]
Hence, by the minimality of $\m_0$ we have that $\mathcal{E} (\tmmathbf{u}_0)
=\mathcal{E} (\m_0)$ and, therefore,
\begin{equation}
  \int_{- 1}^1 \int_0^{2 \pi} | \partial_z \m_0 (z, t) |^2 \mathd t \mathd z =
  0,
\end{equation}
from which we conclude that $\m_0$ is $z$-invariant. This completes the proof
when constraint \eqref{eq:condprojnull} is not present. Note that although we assumed that
$\Gamma = \Stwo^1$, the proof works for general smooth Jordan curves $\Gamma$ with minor notational modifications.

The same proof now works even if we assume the weakly axially symmetric
constraint \eqref{eq:condprojnull}. Indeed, if $\m_0^{\bot}$ is a weakly
axially symmetric minimizer, then also $\tmmathbf{u}_0^{\bot}$ is weakly
axially symmetric. To see this, we observe that
\begin{align}
  \left| \int_0^{2 \pi} \tmmathbf{u}_{\varepsilon}^{\bot} (t) -
  \m_{\varepsilon}^{\bot} (z, t) \mathd t \right| & = \left| \int_0^{2 \pi}
  \m_{\varepsilon}^{\bot} (z_{\varepsilon}, t) - \m_{\varepsilon}^{\bot} (z,
  t) \mathd t \right| \nonumber\\
  & = \left| \int_0^{2 \pi} \int_z^{z_{\varepsilon}} \partial_z
  \m_{\varepsilon}^{\bot} (s, t) \mathd s \mathd t - \int_z^{z_{\varepsilon}}
  \partial_z \left( \int_0^{2 \pi} \m_0^{\bot} (s, t) \mathd s \right) \mathd
  t \right| \nonumber\\
  & =  \left| \int_0^{2 \pi} \int_z^{z_{\varepsilon}} \partial_z
  \m_{\varepsilon}^{\bot} (s, t) \mathd s \mathd t - \int_0^{2 \pi}
  \int_z^{z_{\varepsilon}} \partial_z \m_0^{\bot} (s, t) \mathd s \mathd t
  \right| \nonumber\\
  & \leqslant  \int_0^{2 \pi} \int_{- 1}^1 \left| \partial_z \left(
  \m_{\varepsilon}^{\bot} - \m_0^{\bot} \right) \right| \mathd s \mathd t, 
\end{align}
and, therefore, passing to the limit for $\varepsilon \rightarrow 0$, we get
that
\begin{equation}
  \left| \int_0^{2 \pi} \tmmathbf{u}_0^{\bot} (t) - \m_0^{\bot} (z, t) \mathd
  t \right| = 0,
\end{equation}
i.e., $\langle \tmmathbf{u}_0^{\bot} \rangle_{\Stwo^1} = \left\langle
\m_0^{\bot} (z, \cdot) \right\rangle_{\Stwo^1}=\mathbf{0}$. This completes the proof.
\end{proof}

\section{Global minimizers. A sharp Poincar{\'e}-type inequality on the
cylinder}\label{se:SPI}

{\noindent}An exact characterization of the minimizers of the energy
functional \eqref{eq:micromagenfunonC} is a nontrivial task. Qualitative
aspects of the energy landscape have been investigated in \cite{streubel2016magnetism} through numerical simulations. However, sometimes it is enough to obtain a meaningful lower bound on the energy to gain information on the ground states. For that, we relax the constraint from $\m$ being $\Stwo^2$-valued to the following energy constraint:
\begin{equation}
  \frac{1}{4 \pi} \int_{\mathcal{C}} | \m |^2 = 1,
  \label{eq:relaxconstr}
\end{equation}
with $\mathcal{C} \assign I \times \Stwo^1$ and $I \assign [- 1, 1]$. From the
physical point of view, this type of relaxation corresponds to a passage from
classical physics to a probabilistic quantum mechanics perspective, and it has
been proved to be useful in obtaining nontrivial lower bounds of the ground
state micromagnetic energy (see, e.g.,
{\cite{BrownA1968,di2012generalization,Di_Fratta_2019}}). From the
mathematical perspective, replacing the pointwise constraint $| \m |
= 1$ {\tmabbr{a.e.}} in $\mathcal{C}$ with \eqref{eq:relaxconstr} frames the
minimization problem in the context of Poincar{\'e}-type inequalities, where
sometimes the relaxed problem can be solved exactly, and the dependence of the
minimizers on the geometrical and physical properties of the model made
explicit. This relaxation can help to obtain sufficient conditions for
minimizers to have specific geometric structures (see, e.g.,
{\cite{BrownA1968,di2012generalization,Di_Fratta_2019}}). We note that the
pointwise constraint $| \m | = 1$ {\tmabbr{a.e.}} in $\mathcal{C}$
is equivalent to the following two energy constraints in terms of the $L^2$
and $L^4$ norms:
\begin{equation}
  \frac{1}{4 \pi} \int_{\mathcal{C}} | \m |^2 = 1, \qquad \frac{1}{4
  \pi} \int_{\mathcal{C}} | \m |^4 = 1.
\end{equation}
Indeed, by Cauchy--Schwarz inequality $4 \pi = (| \m |^2, 1)_{L^2
(\mathcal{C})} \leqslant \| | \m |^2 \|_{L^2 (\mathcal{C})}  \| 1
\|_{L^2 (\mathcal{C})} = 4 \pi$, which assures that the equality sign in
the previous estimate is reached only when $| \m |^2$ is constant
and, therefore, necessarily equal to $1$. It follows that the relaxed problem
can also be interpreted as the one obtained by forgetting about the $L^4$
constraint.

Our results include the precise characterization of the minimal
value and global minimizers of the energy functional $\mathcal{E}$ defined in
\eqref{eq:micromagenfunonC} on the space of $H^1( \mathcal{C}, \RR^3
)$ vector fields satisfying the relaxed constraint
\eqref{eq:relaxconstr}. Thanks to Proposition \ref{Prop:dimred}, we can focus
on the analysis of the minimizers in $H^1( \Stwo^1, \RR^3)$ of
the {\emph{normalized}} energy functional
\begin{equation}
  \GG (\uu) \assign \frac{1}{2 \pi} \int_{\Stwo^1} \left|
  \grad_{\gamma} \uu \right|^2 \mathd \gamma \; + \;
  \frac{\kappa^2}{2 \pi} \int_{\Stwo^1} | \uu \times \n |^2
  \mathd \gamma, \label{eq:renoen}
\end{equation}
subject to the $L^2$-constraint
\begin{equation}
  \frac{1}{2 \pi} \int_{\Stwo^1} | \uu |^2 \mathd \gamma = 1.
  \label{eq:contru0}
\end{equation} 
Clearly, every minimizer of \eqref{eq:renoen} in $H^1( \Stwo^1, \Stwo^2
)$ satisfies the constraint \eqref{eq:contru0} and provides an upper
bound to the minimal energy associated with the problem
\eqref{eq:renoen}-\eqref{eq:contru0}. Thus, problem
\eqref{eq:renoen}-\eqref{eq:contru0} is a relaxed version of the minimization
problem for $\GG$ in $H^1( \Stwo^1, \Stwo^2)$. Although the
expression of the energy functional \eqref{eq:renoen} is, up to the constant
factor $\frac{1}{2 \pi}$, the same as of $\mathcal{F}$ in
\eqref{eq:redenmain}, we denoted it by $\GG$ to stress that it is part of the
relaxed minimization problem.

The existence of minimizers of the problem
\eqref{eq:renoen}-\eqref{eq:contru0} quickly follows by direct methods in the
Calculus of Variations. However, uniqueness is out of the question due to the
energy's invariance under the orthogonal group and reflections. Indeed, for
every $\kappa^2 > 0$, at least two minimizers always exist because if
$\uu$ is a minimizer of $\GG$, also $-\uu$ minimizes $\GG$.
We only focus on the nontrivial case $\kappa^2 \neq 0$; otherwise, constant
configurations $\tmmathbf{\sigma} \in \Stwo^2$ are the only minimizers.

In what follows, we denote by $\n$ the outward normal vector field
to $\Stwo^1$ and by $\tmmathbf{\tau} \assign \grad_{\gamma} \n$ the
tangential one. When we refer to the local coordinates representation of a
configuration $\uu_{\bot} \in H^1( \Stwo^1, \RR^3)$, it
is always meant the curve $\uu_{\bot} \circ \gamma$, with $\gamma : t
\in [- \pi, \pi] \mapsto (\cos t) e_1 + (\sin t) e_2$, and $t \in [- \pi,
\pi]$. Thus, e.g., in local coordinates, we have that $\tmmathbf{\tau} (t) =
(- \sin t, \cos t)$ and $\n (t) = (\cos t, \sin t)$.

Our main result includes the precise characterization of the minimal value and
global minimizers of the relaxed problem \eqref{eq:renoen}-\eqref{eq:contru0}
on the space of $H^1( \Stwo^1, \RR^3)$. In fact, we establish the
following sharp Poincar{\'e}-type inequality in $H^1( \Stwo^1, \RR^3)$.
\begin{figure}[t]
{\includegraphics[width=\textwidth]{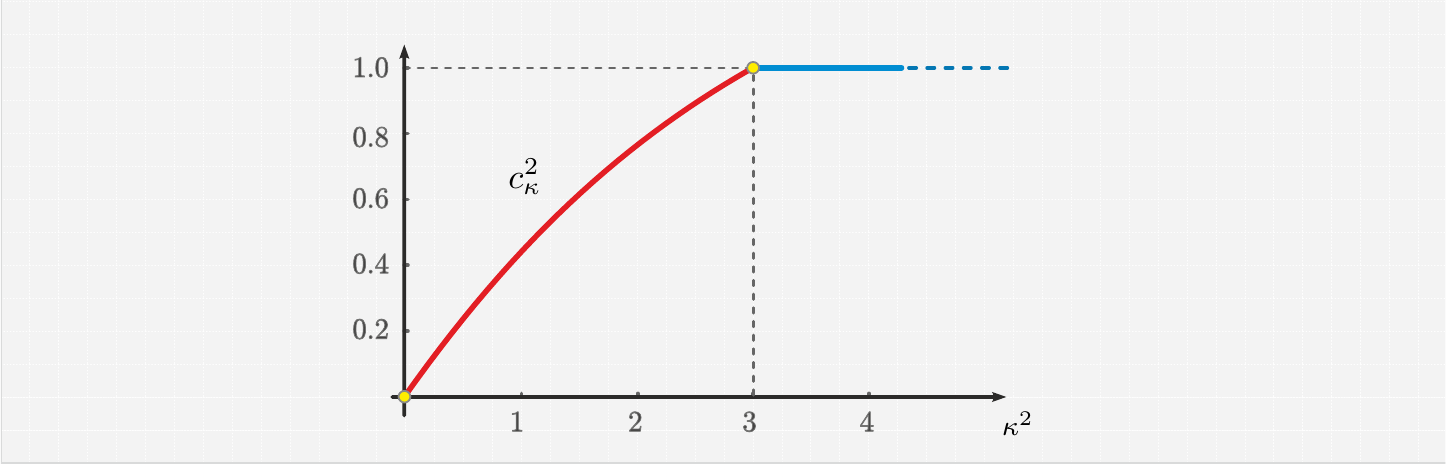}}
  \caption{\label{Fig:optimalPC}The graph of the optimal Poincar{\'e} constant
  $c^2_{\kappa}$ as a function of the parameter $\kappa^2$. The optimal
  constant $c^2_{\kappa}$ increases until the saturation value $c^2_{\kappa} =
  1$ is reached at $\kappa^2 = 3$.
  }
\end{figure}

\begin{theorem}[\tmname{sharp Poincar{\'e}-type inequality in} $H^1( \Stwo^1,\RR^3)$]~\label{thm:PoincIneq}
  For every $\kappa^2 > 0$ the following sharp Poincar{\'e}-type inequality
  holds,
  \begin{equation}
    \int_{\Stwo^1} \left| \grad_{\gamma} \uu \right|^2 \mathd \gamma + \kappa^2 
    \int_{\Stwo^1} | \uu \times \n |^2 \mathd \gamma \hspace{0.27em} \geqslant
    \hspace{0.27em} c^2_{\kappa}  \int_{\Stwo^1} | \uu |^2 \mathd \gamma \quad
    \forall \uu \in H^1 (\Stwo^1, \RR^3), \label{eq:PoinconS1R3}
  \end{equation}
  where the best Poincar{\'e} constant $c^2_{\kappa}$ is the continuous function
  of $\kappa$ given by
  \begin{equation}
    c_{\kappa}^2 \assign \left\{ \begin{array}{lll}
      1 & \text{if } & \kappa^2 \geqslant 3,\\
      \frac{1}{2} (\kappa^2 - \omega_{\kappa}^2 + 4) & \text{if } & 0 < \kappa^2
      \leqslant 3,
    \end{array} \right. \label{eq:bestconsts}
  \end{equation}
  with $\omega_{\kappa}^2 \assign \sqrt{\kappa^4 + 16}$. Moreover, the equality
  sign in the Poincar{\'e} inequality {\emph{\eqref{eq:PoinconS1R3}}} is reached
  if, and only if, $\uu \in H^1 (\Stwo^1, \RR^3)$ has the following expressions:
  \begin{enumerate}
    \item[{\emph{i.}}] If $\kappa^2 > 3$, the equality sign in
    {\emph{\eqref{eq:PoinconS1R3}}} is reached only by the normal vector fields
    $\pm \n$.
    
    \item[{\emph{ii.}}] If $\kappa^2 = 3$, the equality is reached if, and only
    if, $\uu$ is an element of the family represented in local coordinates by
    \begin{equation}
      \uu_{\bot} (t) = \sqrt{2} \rho_1 \cos (\theta + t) \tmmathbf{\tau} (t) +
      \left[ \pm \sqrt{1 - 5 \rho_1^2} + 2 \sqrt{2} \rho_1 \sin (\theta + t)
      \right] \n (t) \label{eq:newminsk23thm}
    \end{equation}
    for arbitrary $\theta \in [- \pi, \pi]$ and $0 \leqslant \rho_1 \leqslant 1
    / \sqrt{5}$. In particular ($\rho_1 = 0$), the normal vector fields $\pm \n$
    persist as minimizers for $\kappa^2 = 3$.
    
    \item[{\emph{iii.}}] If $0 < \kappa^2 < 3$, the equality sign in
    {\emph{\eqref{eq:PoinconS1R3}}} is reached by any element of the family
    represented in local coordinates by
    \begin{equation}
      \uu_{\bot} (t) = \sqrt{2} (\sin \phi_{\kappa}) \cos (\theta + t)
      \tmmathbf{\tau} (t) + \sqrt{2} (\cos \phi_{\kappa}) \sin (\theta + t) \n
      (t), \label{eq:newminsk24thm}
    \end{equation}
    with $\theta \in [- \pi, \pi]$ arbitrary, and $\phi_{\kappa} \in [0, \pi /
    2]$ given by
    \begin{equation}
      \phi_{\kappa} = \frac{1}{2} \arctan (4 / \kappa^2) .
      \label{eq:eqsforphikthm}
    \end{equation}
    Moreover, there are no $\Stwo^2$-valued configurations for which the
    equality sign is achieved in {\emph{\eqref{eq:PoinconS1R3}}}.
  \end{enumerate}
  The normal fields $\pm \n$ are universal configurations as their energy does
  not depend on $\kappa^2$.
\end{theorem}

\begin{remark}
  In view of our original problem
  concerning $\Stwo^2$-valued minimizers, we note that setting $\rho_1 = 0$ in \eqref{eq:newminsk23thm} we recover the
  normal vector fields $\pm \n$, and these are the only
  $\Stwo^2$-valued minimizers of the problem \eqref{eq:renoen}-\eqref{eq:contru0} when $\kappa^2 = 3$. Instead, when $0 <
  \kappa^2 < 3$, there are no $\Stwo^2$-valued configurations for which the
  equality sign is achieved in \eqref{eq:PoinconS1R3}.
\end{remark}

A graph of the optimal Poincar{\'e} constant $c^2_{\kappa}$ as a function of
$\kappa^2$ is depicted in Figure~\ref{Fig:optimalPC}. Figure~\ref{fig:k2eq3}represents a plot of the minimal configurations
in \eqref{eq:newminsk23thm} for which the equality sign is attained in the
Poincar{\'e} inequality when $\kappa^2 = 3$. Also, a plot of the minimal
vector fields in \eqref{eq:newminsk24thm} is given in Figure~\ref{fig:klet3}
for different values of $0 < \kappa^2 < 3$.

Before giving the proof of Theorem~\ref{thm:PoincIneq}, we want to point out
some of its consequences.

\begin{proposition}
  For every $\kappa^2 > 0$, the map
  \[ \uu \in H^1( \Stwo^1, \RR^3) \mapsto \left(
     \int_{\Stwo^1} \left| \grad_{\gamma} \uu \right|^2 \mathd \gamma
     + \kappa^2 | \uu \times \n |^2 \mathd \gamma
     \right)^{1 / 2} \]
  is a norm on $H^1( \Stwo^1, \RR^3)$ equivalent to the classical
  norm
  \[ \| \uu \|_{H^1( \Stwo^1, \RR^3)} = \left(
     \int_{\Stwo^1} \left| \grad_{\gamma} \uu \right|^2 \mathd \gamma
     + | \uu |^2 \mathd \gamma \right)^{1 / 2} . \]
\end{proposition}
\begin{figure}[t]
    {\includegraphics[width=\textwidth]{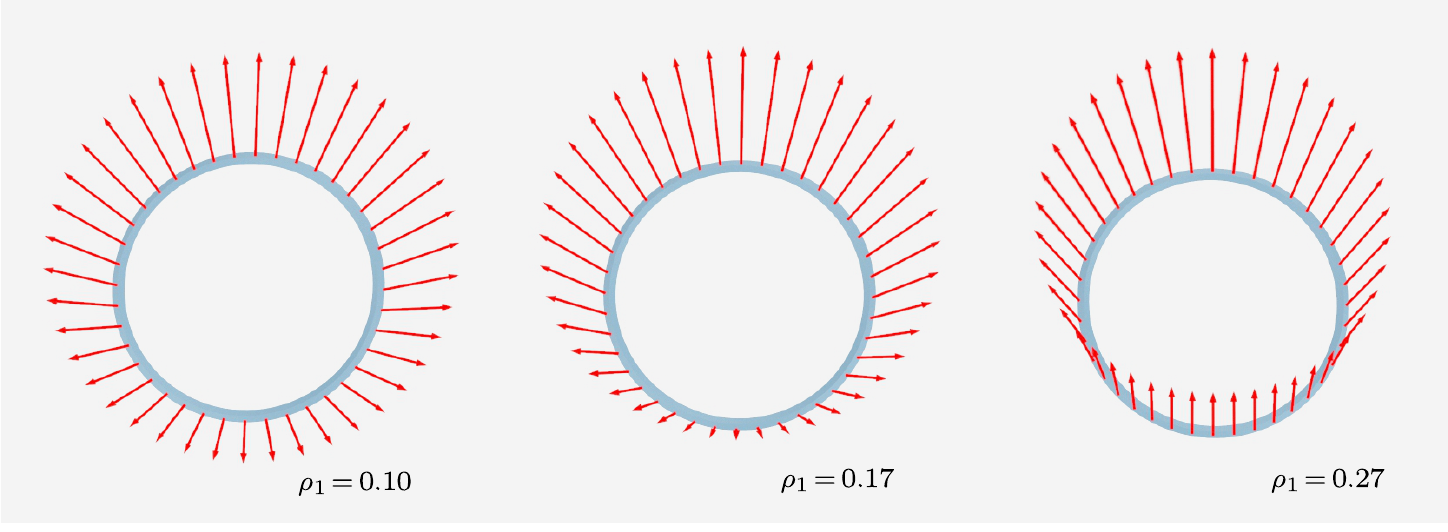}}
    \caption{\label{fig:k2eq3}A plot of the minimal configurations in
    \eqref{eq:newminsk23thm} for which the equality sign is attained in the
    Poincar{\'e} inequality when $\kappa^2 = 3$. The optimal vector fields are
    represented for three different values of $\rho_1$. From left to right, we
    plot \eqref{eq:newminsk23thm} for $\rho_1 = 0.10$, $\rho_1 = 0.17$ and
    $\rho_1 = 0.27$.
    }
  \end{figure}
Next, by Theorem~\ref{thm:PoincIneq} we get that for $\kappa^2 \geqslant 3$
the relaxed minimization problem \eqref{eq:renoen}-\eqref{eq:contru0} admits
$\Stwo^2$-valued minimizers. Thus, as a byproduct of
Theorem~\ref{thm:PoincIneq}, we obtain the following characterization of
micromagnetic ground states in thin cylindrical shells.

\begin{theorem}[\tmname{Micromagnetic ground states in thin cylindrical
  shells}]\label{thm:BFTCylinder} For every value $\kappa^2 > 0$ of the
  anisotropy, the normal vector fields $\pm \n$, as well as the
  constant vector fields $\pm \e_3$ are stationary points of the micromagnetic
  energy functional
  {\opt}{\tmabbr{{\emph{cf.}}}}~{\emph{\eqref{eq:micromagenfunonC}{\cpt}}}
  \[ \mathcal{E} \left( \m \right) = \int_{\mathcal{C}} \left| \grad_{\xi}
     \m \right|^2 \mathd \xi + \kappa^2 \int_{\mathcal{C}} |
     \m \times \n |^2 \mathd \xi, \quad \m \in H^1(
     \mathcal{C}, \Stwo^2), \]
  and the following properties hold:
  \begin{enumerate}
    \item[\emph{i.}] If $\kappa^2 \geqslant 3$, the normal vector fields $\pm
    \n$ are the only global minimizers of the energy functional
    $\mathcal{E}$ in $H^1( \mathcal{C}, \Stwo^2)$. Also, they are
    locally stable for every $\kappa^2 \geqslant 1$ and unstable for $0 <
    \kappa^2 < 1$. Moreover, when $\kappa^2 > 1$, the normal vector fields
    $\pm \n$ are local minimizers of the energy $\mathcal{E}$.
    
    \item[\emph{ii.}] The constant vector fields $\pm \e_3$ are unstable for all $\kappa^2
    > 0$.
  \end{enumerate}
\end{theorem}

\begin{remark}
  Although the constant vector fields $\pm \e_3$ are unstable for all
  $\kappa^2 > 0$, they are stable in the class of axially symmetric
  minimizers, as has been shown in Theorem~\ref{cor:axsymmmin}. 
\end{remark}
 \begin{figure}[t]
    {\includegraphics[width=\textwidth]{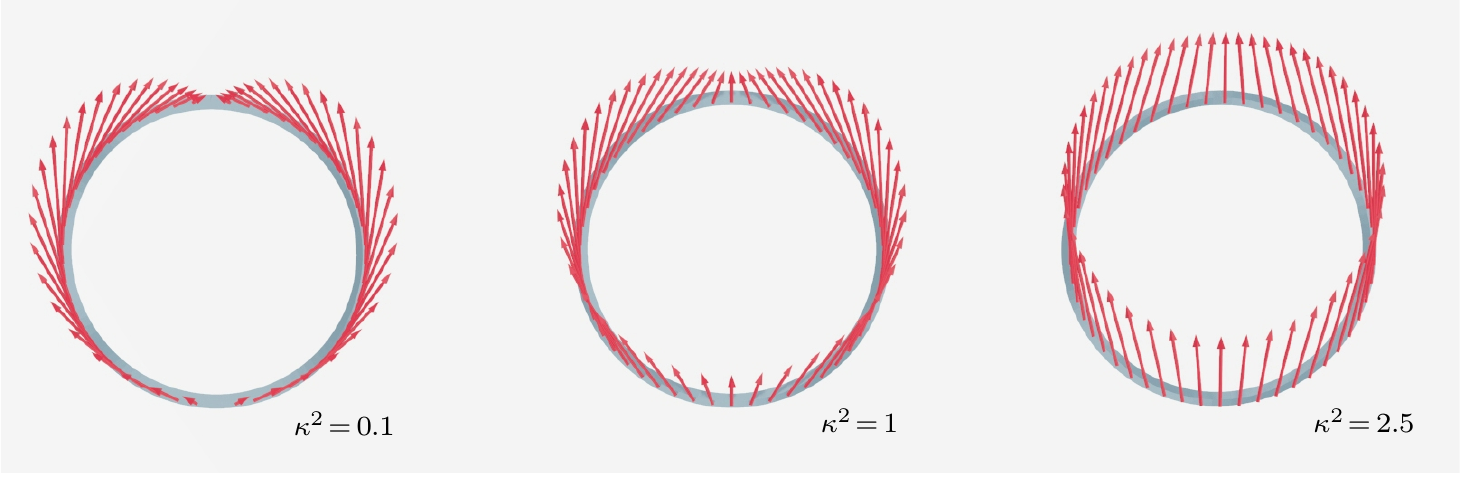}}
    \caption{\label{fig:klet3}A plot of the minimal vector fields in
    \eqref{eq:newminsk24thm} for which the equality sign is attained in the
    Poincar{\'e} inequality when $0 < \kappa^2 < 3$. From left to right, we
    plot \eqref{eq:newminsk24thm} for $\kappa^2 = 0.1$, $\kappa^2 = 1$, and
    $\kappa^2 = 2.5$.
    }
  \end{figure}

\begin{proof}
  In coordinates, the energy functional ({\tmabbr{cf.}}~\eqref{eq:reden})
  reads as
  \begin{equation}
    \mathcal{F} (\m) = \int_{- 1}^1 \int_{- \pi}^{\pi} \left|
    \grad \m \right|^2 +
    \kappa^2 \left| \m \times \n \right|^2  \hspace{0.17em} \mathd t \, \mathd z,
  \end{equation}
  with $\m \in H^1 (\mathcal{C}, \Stwo^2)$ and
  $\grad = (\partial_z,
  \partial_t)$. The Euler-Lagrange associated with $\mathcal{F}$ read as
  \begin{equation}
    - \Delta \m - \kappa^2  (\m \cdot \n) \n = \lambda (t)
    \m, \qquad \lambda (t) = |
    \grad \m |^2 - \kappa^2
    (\m \cdot \n)^2 .
  \end{equation}
  It is therefore clear that $\m (t) = \pm \e_3$ and $\m (t) = \pm \n (t)$ are
  critical points of the above energy. Let us investigate their stability. The
  second variation of $\mathcal{F}$ is given by
  \begin{equation}
    \mathcal{F}'' (\m) (\tmmathbf{\varphi}) = \int_{- 1}^1 \int_{- \pi}^{\pi}
    | \grad
    \tmmathbf{\varphi}|^2 - \kappa^2  (\tmmathbf{\varphi} \cdot \n)^2 - (|
    \grad \m |^2 - \kappa^2
    (\m \cdot \n)^2) |\tmmathbf{\varphi}|^2 \mathd t \, \mathd z,
  \end{equation}
  and is defined on all $\tmmathbf{\varphi} \in H^1\cap L^\infty (\mathcal{C},
  \RR^3)$ such that $\tmmathbf{\varphi} \cdot \m
  \equiv 0$ in $\mathcal{C} $.
  
  {\noindent}{\emph{Proof of i.}} It is clear from Theorem~\ref{thm:PoincIneq}
  that $\pm \tmmathbf{n}$ are the only global minimizers of $\mathcal{F}$
  whenever $\kappa^2 \geqslant 3$. It remains to analyze their stability in
  the range $0 < \kappa^2 < 3$. For that, it is sufficient to evaluate the
  second variation at $\m = \pm \tmmathbf{n}$, which reads as
  \begin{equation}
    \mathcal{F}'' (\pm \n) (\tmmathbf{\varphi}) = \int_{- 1}^1 \int_{-
    \pi}^{\pi} | \grad
    \tmmathbf{\varphi}|^2 + (\kappa^2 - 1) |\tmmathbf{\varphi}|^2 \mathd t \,
    \mathd z. \label{eq:secvarF}
  \end{equation}
  Setting $\tmmathbf{\varphi}= \e_3$ we obtain that $\mathcal{F}'' (\pm \n)
  (\e_3) = 4 \pi (\kappa^2 - 1)$. Therefore the normal vector fields $\pm \n$
  are unstable critical points of $\mathcal{F}$ when $\kappa^2 < 1$.
  
  However, \eqref{eq:secvarF} shows that for $\kappa^2 > 1$ we have uniform
  local stability of the critical points $\pm \n$. We want to show that $\pm
  \n$ are, in fact, local minimizers of the energy functional $\mathcal{F}$.
  We focus on $+ \n$, the argument for $- \n$ being the same. Following
  {\cite{Di_Fratta_2019}}, first, we compute the finite variation of
  $\mathcal{F}$ corresponding to an admissible increment $\tmmathbf{v}$ of
  $\n$, i.e., $\vv \in H^1 (\mathcal{C}, \RR^3)$
  such that $\left| \n + \vv \right| = 1$ in $\mathcal{C}$. A simple
  computation gives
  \begin{equation}
    \mathcal{F} (\n + \vv) -\mathcal{F} (\n) = \int_{- 1}^1 \int_{- \pi}^{\pi}
    \left| \grad \vv
    \right|^2 + \kappa^2 \left( \left| \vv \right|^2 - (\tmmathbf{v} \cdot
    \n)^2 \right) - \left| \vv \right|^2 \hspace{0.17em} \mathd t \, \mathd z.
    \label{eq:tempv1}
  \end{equation}
  Next, we define $\tmmathbf{\varphi}=\tmmathbf{v}- (\tmmathbf{v} \cdot \n)
  \n$ and want to rewrite the previous expression in $\tmmathbf{\varphi}$.
  Since $\tmmathbf{v}=\tmmathbf{\varphi}+ (\tmmathbf{v} \cdot \n) \n$ we have
  $\partial_t \tmmathbf{v}= \partial_t \tmmathbf{\varphi}+ (\tmmathbf{v} \cdot
  \n) \tmmathbf{\tau}+ \left( \partial_t (\tmmathbf{v} \cdot \n) \right) \n$
  with $\tmmathbf{\tau} \assign \partial_t \n$. Using that $\partial_t
  \tmmathbf{\varphi} \cdot \n = -\tmmathbf{\varphi} \cdot \tmmathbf{\tau}$
  because of $\tmmathbf{\varphi} \cdot \n = 0$, we obtain
  \begin{equation}
    | \partial_t \tmmathbf{v} |^2 = | \partial_t \tmmathbf{\varphi} |^2 +
    (\tmmathbf{v} \cdot \n)^2 + \left| \partial_t (\tmmathbf{v} \cdot \n)
    \right|^2 + 2 (\partial_t \tmmathbf{\varphi} \cdot \tmmathbf{\tau})
    (\tmmathbf{v} \cdot \n) - 2 (\tmmathbf{\varphi} \cdot \tmmathbf{\tau})
    \partial_t (\tmmathbf{v} \cdot \n) .
  \end{equation}
  Integrating by parts the previous relation, we infer that
  \begin{equation}
    \int_{- 1}^1 \int_{- \pi}^{\pi} | \partial_t \tmmathbf{v} |^2 \mathd t
    \mathd z = \int_{- 1}^1 \int_{- \pi}^{\pi} | \partial_t \tmmathbf{\varphi}
    |^2 + (\tmmathbf{v} \cdot \n)^2 + \left| \partial_t (\tmmathbf{v} \cdot
    \n) \right|^2 + 4 (\partial_t \tmmathbf{\varphi} \cdot \tmmathbf{\tau})
    (\tmmathbf{v} \cdot \n) \mathd t \, \mathd z. \label{eq:exprdzf2rev}
  \end{equation}
  Similarly, we have
  \begin{equation}
    | \partial_z \tmmathbf{\varphi} |^2 = \left| \partial_z \tmmathbf{v}-
    (\partial_z \tmmathbf{v} \cdot \n) \n \right|^2 = | \partial_z
    \tmmathbf{v} |^2 - \left| \partial_z \left( \tmmathbf{v} \cdot \n \right)
    \right|^2 . \label{eq:exprdzf2}
  \end{equation}
  Hence, plugging the previous expression into \eqref{eq:tempv1}, using that
  $| \tmmathbf{\varphi} |^2 = | \tmmathbf{v} |^2 - (\tmmathbf{v} \cdot \n)^2$,
  and taking into account \eqref{eq:secvarF}, we obtain
  \begin{eqnarray}
    \mathcal{F} (\n + \vv) -\mathcal{F} (\n) & = & \int_{- 1}^1 \int_{-
    \pi}^{\pi} \left| \grad
    \vv \right|^2 + \kappa^2 | \tmmathbf{\varphi} |^2 - \left| \vv \right|^2 
    \hspace{0.17em} \mathd t \, \mathd z \nonumber\\
    & \overset{\eqref{eq:exprdzf2rev}}{=} & \int_{- 1}^1 \int_{- \pi}^{\pi} |
    \grad
    \tmmathbf{\varphi} |^2 + (\kappa^2 - 1) | \tmmathbf{\varphi} |^2 + \left|
    \grad (\tmmathbf{v}
    \cdot \n) \right|^2 + 4 (\partial_t \tmmathbf{\varphi} \cdot
    \tmmathbf{\tau}) (\tmmathbf{v} \cdot \n) \mathd t \, \mathd z \nonumber\\
    & \overset{\eqref{eq:secvarF}}{=} & \mathcal{F}'' (\n)
    (\tmmathbf{\varphi}) + \int_{- 1}^1 \int_{- \pi}^{\pi} \left|
    \grad (\tmmathbf{v}
    \cdot \n) \right|^2 + 4 (\partial_t \tmmathbf{\varphi} \cdot
    \tmmathbf{\tau}) (\tmmathbf{v} \cdot \n) \mathd t \, \mathd z. 
    \label{eq:secvt2}
  \end{eqnarray}
  Note that since $- 2\tmmathbf{v} \cdot \n = |\tmmathbf{v}|^2$ and
  $\tmmathbf{v}=\tmmathbf{\varphi}+ (\tmmathbf{v} \cdot \n) \n$ we have that
  $|\tmmathbf{v}|^2 = |\tmmathbf{\varphi}|^2 + \frac{1}{4} |\tmmathbf{v}|^4$.
  Overall, from the previous considerations and \eqref{eq:secvt2}, we obtain
  that
  \begin{align}
    \mathcal{F} (\n + \vv) -\mathcal{F} (\n) &\;  = \; \mathcal{F}'' (\n)
    (\tmmathbf{\varphi}) + \int_{- 1}^1 \int_{- \pi}^{\pi} \left|
    \grad (\tmmathbf{v}
    \cdot \n) \right|^2 - 2 (\partial_t \tmmathbf{\varphi} \cdot
    \tmmathbf{\tau}) | \tmmathbf{v} |^2 \mathd t \, \mathd z \notag \\
    &\; \geqslant \; \mathcal{F}'' (\n) (\tmmathbf{\varphi}) - \int_{- 1}^1
    \int_{- \pi}^{\pi} \left( 2 | \tmmathbf{v} |^4 + \frac{1}{2} | \partial_t
    \tmmathbf{\varphi} |^2 \right) \mathd t \, \mathd z \notag \\
    &\; \geqslant \; \frac{1}{2} \|
    \grad
    \tmmathbf{\varphi}\|_{L^2 (\mathcal{C})}^2 + (\kappa^2 - 1)
    \|\tmmathbf{v}\|_{L^2 (\mathcal{C})}^2 - \left( 2 + \frac{\kappa^2 - 1}{4}
    \right) \|\tmmathbf{v}\|_{L^4 (\mathcal{C})}^4 . 
  \end{align}
  Now we use Gagliardo--Nirenberg inequality (see, e.g.,
  {\cite{fiorenza2018detailed}}), i.e., the existence of a positive constant
  $c_L > 0$ such that $\|\tmmathbf{v}\|_{L^4 (\mathcal{C})}^4 \leqslant c_L
  \|\tmmathbf{v}\|_{H^1 (\mathcal{C})}^2 \|\tmmathbf{v}\|_{L^2
  (\mathcal{C})}^2$ for every $\tmmathbf{v} \in H^1 (\mathcal{C},
  \RR^3)$. Therefore
  \begin{equation}
    \mathcal{F} (\n + \vv) -\mathcal{F} (\n) \; \geqslant \; \frac{1}{2} \|
    \grad
    \tmmathbf{\varphi}\|_{L^2 (\mathcal{C})}^2 + \left[ (\kappa^2 - 1) - c_L
    \left( 2 + \frac{\kappa^2 - 1}{4} \right) \|\tmmathbf{v}\|_{H^1
    (\mathcal{C})}^2 \right] \|\tmmathbf{v}\|_{L^2 (\mathcal{C})}^2 . 
  \end{equation}
  To conclude, we observe that for a given $\kappa^2 > 1$ the previous
  right-hand side is nonnegative as soon as $\|\tmmathbf{v}\|_{H^1
  (\mathcal{C})}^2$ is chosen small enough.
  
  {\noindent}{\emph{Proof of ii.}} Testing the second variation at $\m = \pm
  \e_3$ against $\vv = \e_1$ we obtain
  \begin{equation}
    \mathcal{F}'' (\pm \e_3) (\e_1) = - \kappa^2  \int_{- 1}^1 \int_{-
    \pi}^{\pi} (\e_1 \cdot \n)^2 \mathd t \, \mathd z = - 2\pi \kappa^2 < 0.
  \end{equation}
  Therefore we know that the constant
  $\Stwo^2$-valued vector fields $\m = \pm \e_3$
  are unstable for all $\kappa^2 > 0$.
\end{proof}

\begin{remark}
  As already pointed out in the introduction, there are several analogies in
  the behavior of the minimizers of the micromagnetic energy in cylindrical
  and spherical surfaces. However, there are also remarkable differences.
  Indeed, in both cases, the normal vector fields are the unique
  global minimizers of the energy functional in a wide range of
  parameters~{\cite{Di_Fratta_2019}}. However, the topological consequences
  are very different. On the one hand, the normal vector fields $\pm
  \n_{\Stwo^2}$ to $\Stwo^2$ carry a different skyrmion number
  because $\jdeg \left( \pm \n_{\Stwo^2} \right) = \pm 1$, and, by
  Hopf theorem {\cite{milnor1997topology}}, they cannot be homotopically
  mapped from one to the other. This translates into the so-called topological
  protection of the ground states. On the other hand, due to the odd
  dimension, the two normal vector fields $\pm \n$ to $\Stwo^1$ have the same
  degree, and therefore, they can be ``easily'' switched from one to the other
  through appropriate external fields.
\end{remark}

\begin{remark}
  The result of Theorem~\ref{thm:BFTCylinder} can be interpreted in terms of
  the size of the magnetic particle. Indeed, a simple rescaling of the energy
  functional \eqref{eq:micromagenfunonC} shows that the range $\kappa^2
  \geqslant 3$ corresponds to the geometric regime of cylindrical magnets with
  a large radius. This is the dual version of Brown's fundamental theorem on
  $3 d$ fine ferromagnetic particles, which shows the existence of a critical
  diameter below which the unique energy minimizers are the constant-in-space
  magnetizations~{\cite{BrownA1968,di2012generalization,alouges2015liouville}}.
\end{remark}

\subsection{Proof of the sharp Poincar{\'e} inequality
(Theorem~\ref{thm:PoincIneq})}To prove Theorem~\ref{thm:PoincIneq}, we need
the following result which, in particular, shows that the constant vector
field $\tmmathbf{e}_3 \in \Stwo^2$ is never a global minimizer of the relaxed
minimization problem \eqref{eq:renoen}-\eqref{eq:contru0}. In fact, any
{\emph{critical point}} of $\mathcal{G}$ with energy strictly less than
$\kappa^2$ (in particular, any minimizer) is in-plane.

\begin{lemma}
  \label{lemma:e3eq0}Let $\kappa^2 \neq 0$. Any critical point $\uu
  \in H^1( \Stwo^1, \RR^3)$ of the relaxed problem
  {\emph{\eqref{eq:renoen}-\eqref{eq:contru0}}} satisfying the energy bound
  $\GG (\uu) < \kappa^2$ is in-plane, i.e.,
  \begin{equation}
    \uu (\gamma) \cdot \tmmathbf{e}_3 = 0 \quad \forall \gamma \in
    \Stwo^1 . \label{eq:u3eq0}
  \end{equation}
  Moreover, the minimal energy satisfies the energy bounds
  \begin{equation}
    0 < \GG (\uu) \leqslant \min \left\{ \frac{\kappa^2}{2}, 1
    \right\} \label{eq:ubene0} .
  \end{equation}
  In particular, every minimizer of the relaxed problem
  {\emph{\eqref{eq:renoen}-\eqref{eq:contru0}}} is in-plane.
\end{lemma}

\begin{proof}
  In terms of the standard parameterization of the unit circle $\Stwo^1$, the
  problem reduces to the minimization, in the Sobolev space of periodic
  functions $H^1_{\sharp} ( [- \pi, \pi], \RR^3)$, of the energy
  functional
  \begin{equation}
    \GG (\uu) \assign \frac{1}{2 \pi} \int_{- \pi}^{\pi} | \partial_t
    \uu |^2 \mathd t \; + \; \frac{\kappa^2}{2 \pi} \int_{-
    \pi}^{\pi} | \uu \times \n |^2 \mathd t,
    \label{eq:FscriptforPoinc}
  \end{equation}
  subject to the constraint
  \begin{equation}
    \frac{1}{2 \pi} \int_{- \pi}^{\pi} | \uu |^2 \mathd t = 1 .
    \label{eq:contru}
  \end{equation}
  We start observing that as soon as $\kappa^2 \neq 0$ the minimal energy is
  {\emph{strictly}} positive. Indeed, if $\GG (\uu) = 0$, then one
  gets at the same time $\uu=\n$ and $\uu \assign
  \tmmathbf{\sigma}$ for some $\tmmathbf{\sigma} \in \Stwo^2$ (because of the
  constraint \eqref{eq:contru}). Thus $\GG (\uu) > 0$ for every
  $\uu \in H^1_{\sharp}( [- \pi, \pi], \RR^3)$
  satisfying the constraint \eqref{eq:contru}. In fact, the minimum in the
  class of constant $\Stwo^2$-valued configurations is reached by any element
  of the class of in-plane uniform fields, i.e., by any configuration
  $\tmmathbf{\sigma} \in \Stwo^2$ such that $\tmmathbf{\sigma} \cdot
  \tmmathbf{e}_3 = 0$. The common minimum value in this class being
  \begin{equation}
    \GG (\tmmathbf{\sigma}) = \frac{\kappa^2}{2 \pi} \int_{- \pi}^{\pi} |
    \tmmathbf{\sigma} |^2 - (\tmmathbf{\sigma} \cdot \n)^2 \mathd t
    = \frac{\kappa^2}{2} > 0.
  \end{equation}
  Also, note that since $| \partial_t \n |^2 = 1$ we have $\GG
  (\n) = 1$ regardless of the value of $\kappa$. Therefore, if
  $\uu$ is a minimum point of the relaxed minimization problem
  \eqref{eq:FscriptforPoinc}-\eqref{eq:contru} then
  \begin{equation}
    0 < \GG (\uu) \leqslant \min \left\{ \frac{\kappa^2}{2}, 1
    \right\} .
  \end{equation}
  This proves \eqref{eq:ubene0}.
  
  Next, we consider the Euler--Lagrange equations associated with the relaxed
  minimization problem \eqref{eq:FscriptforPoinc}-\eqref{eq:contru}. A direct
  computation yields that if $\uu \in H^1_{\sharp} ( [- \pi,
  \pi], \RR^3)$ is a minimizer, then it satisfies the equations
  \begin{equation}
    - \partial_{t  t} \uu+ \kappa^2 (\uu- (\n
    \otimes \n) \uu) = \lambda \uu \quad
    \text{in } \left( H^1_{\sharp} ( [- \pi, \pi], \RR^3 )
    \right)', \quad \frac{1}{2 \pi} \int_{- \pi}^{\pi} | \uu |^2
    \mathd t = 1,  \label{eq:ELFu}
  \end{equation}
  for some Lagrange multiplier $\lambda \in \RR$. To ease notation, we write
  the Euler-Lagrange equations \eqref{eq:ELFu} in their (distributional) form;
  however, behind the scenes, we always mean the associated weak formulation
  in $H^1_{\sharp} ( [- \pi, \pi], \RR^3)$. To get \eqref{eq:ELFu}
  it is sufficient to note that in $H^1_{\sharp} ( [- \pi, \pi], \RR^3
 )$ one has
  \begin{align}
    | \uu \times \n+ \varepsilon \tmmathbf{\varphi} \times
    \n |^2 - | \uu \times \n |^2 +\mathcal{O}
    (\varepsilon^2) =\; & 2 \varepsilon (\uu \times \n)
    \cdot (\tmmathbf{\varphi} \times \n) \\
     =\; & 2 \varepsilon \tmmathbf{\varphi} \cdot (\uu-
    (\n \otimes \n) \uu) . 
  \end{align}
  To get an explicit expression of the Lagrange multiplier $\lambda \in \RR$
  we dot multiply by $\uu$ both members of \eqref{eq:ELFu}. Taking
  into account the constraint \eqref{eq:contru}, we get that
  \begin{equation}
    \frac{1}{2 \pi} \int_{- \pi}^{\pi} | \partial_t \uu |^2 +
    \kappa^2 (| \uu |^2 - (\uu_{\bot} \cdot \n)^2)
    \mathd t = \frac{\lambda}{2 \pi} \int_{- \pi}^{\pi} | \uu
    |^2 \mathd t, 
  \end{equation}
  from which the relation $\lambda = \GG (\uu)$ follows. Combining
  this observation with \eqref{eq:ubene0}, we get that if $\uu$ is a
  minimizer of the constrained minimization problem
  \eqref{eq:FscriptforPoinc}-\eqref{eq:contru} then necessarily
  \begin{equation}
    0 < \lambda \equiv \GG (\uu) \leqslant \min \left\{
    \frac{\kappa^2}{2}, 1 \right\} \label{eq:ubene} .
  \end{equation}
  To conclude the proof it is sufficient to prove that every solution $\uu$ of the Euler--Lagrange equations 
  \eqref{eq:ELFu} such that $ \GG (\uu)<\kappa^2$ is an in-plane configuration. For that, we dot multiply by
  $\tmmathbf{e}_3$ on both sides of \eqref{eq:ELFu} to get the relation
  \begin{equation}
    \partial_{t t} u_3 = \mu \cdot u_3, \quad \text{with } \mu = \kappa^2 -
    | \lambda | .
  \end{equation}
  Note that $\mu = | \mu | > 0$ whenever $\uu$ satisfies the energy
  bound $\GG (\uu) < \kappa^2$. In this case, the general solution of
  the previous equation is given by
  \begin{equation}
    u_3 (t) \assign c_1 e^{t \sqrt{| \mu |}} + c_2 e^{- t \sqrt{| \mu |}},
    \quad c_1, c_2 \in \RR .
  \end{equation}
  The only periodic solution of the previous equation is the zero solution.
  Therefore, any {\emph{critical point}} $\uu \in H^1_{\sharp} (
  [- \pi, \pi], \RR^3)$ of the problem
  \eqref{eq:FscriptforPoinc}-\eqref{eq:contru} satisfying the energy bound
  $\GG (\uu) < \kappa^2$ is in-plane. In particular, because of
  \eqref{eq:ubene}, any minimizer of the problem
  \eqref{eq:FscriptforPoinc}-\eqref{eq:contru} is in-plane. This concludes the
  proof.
\end{proof}

\begin{proof}{Proof of Theorem \ref{thm:PoincIneq}.}
  To deduce the sharp Poincar{\'e} inequality \eqref{eq:PoinconS1R3} one has
  to solve the minimization problem for $\GG$ given by
  \eqref{eq:renoen}-\eqref{eq:contru0}. Indeed, let $\uu_{\ast} \in
  H^1( \Stwo^1, \RR^3)$ be a minimizer of $\GG$ in the class of
  vector fields $\uu \in H^1( \Stwo^1, \RR^3)$ that
  satisfy the $L^2$-constraint $\| \uu \|^2_{L^2 ( \Stwo^1,
  \RR^3)} = 2 \pi$, and set $c^2_{\kappa} \assign \GG
  (\uu_{\ast})$. For every $\tmmathbf{v} \in H^1( \Stwo^1,
  \RR^3 ) \setminus \{ 0 \}$ the configuration
  \begin{equation}
    \uu \assign \sqrt{2 \pi} \frac{\tmmathbf{v}}{\| \tmmathbf{v}
    \|_{L^2 ( \Stwo^1, \RR^3 )}}
  \end{equation}
  is an admissible competitor of the minimization problem for $\GG$ in
  \eqref{eq:renoen}-\eqref{eq:contru0}. Therefore
  \begin{equation}
    \GG (\uu) = \frac{1}{\| \tmmathbf{v} \|_{L^2 ( \Stwo^1,
    \RR^3 )}^2} \left( \int_{\Stwo^1} \left| \grad_{\gamma} \tmmathbf{v}
    \right|^2 \mathd \gamma + \kappa^2 \int_{\Stwo^1} | \tmmathbf{v} \times
    \n |^2 \mathd \gamma \right) \geqslant \GG (\uu_{\ast})
    = c_{\kappa}^2 .
  \end{equation}
  Multiplying all sides of the previous relation by $\| \tmmathbf{v} \|_{L^2
  ( \Stwo^1, \RR^3 )}^2$ we get the sharp Poincar{\'e} inequality
  \eqref{eq:PoinconS1R3} with $c^2_{\kappa}$ being the minimal energy of $\GG$
  subject to \eqref{eq:contru0}. The equality sign is achieved by any
  minimizer of the relaxed problem for $\GG$ in
  \eqref{eq:renoen}-\eqref{eq:contru0}.
  
  According to Lemma~\ref{lemma:e3eq0}, it is possible to restrict our
  attention to the class of vector fields in $\uu_{\bot} \in H^1
  ( \Stwo^1, \RR^2)$ satisfying the $L^2$-constraint
  \begin{equation}
    \frac{1}{2 \pi} \int_{\Stwo^1} | \uu_{\bot} |^2 \mathd \gamma =
    1, \label{eq:constrR2}
  \end{equation}
  and the minimization problem \eqref{eq:renoen}-\eqref{eq:contru0} reduces to
  the minimization of
  \begin{equation}
    \GG (\uu_{\bot}) \assign \frac{1}{2 \pi} \int_{\Stwo^1} \left|
    \grad_{\gamma} \uu_{\bot} \right|^2 \mathd \gamma +
    \frac{\kappa^2}{2 \pi} \int_{\Stwo^1} (| \uu_{\bot} |^2 -
    (\uu_{\bot} \cdot \n)^2) \mathd \gamma
    \label{eq:enFuperp}
  \end{equation}
  in $H^1_{\sharp} ( \Stwo^1, \RR^2)$ subject to the constraint
  \eqref{eq:constrR2}. To solve this minimization problem, we use Fourier
  analysis, and we work in local coordinates, i.e., in $H^1_{\sharp} ( [-
  \pi, \pi], \RR^2)$. Note that, in local coordinates, the
  Euler--Lagrange equations \eqref{eq:ELFu} simplify to
  \begin{equation}
    \partial_{t  t} \uu_{\bot} + \kappa^2 (\n \otimes
    \n) \uu_{\bot} = | \mu | \uu_{\bot}, \qquad
    | \mu | \assign \kappa^2 - | \lambda | > \frac{\kappa^2}{2} .
    \label{eq:ELinuperp}
  \end{equation}
  We consider the moving orthonormal frame of $\RR^2$ given by
$(\tmmathbf{\tau}, \n)$, with both $\tmmathbf{\tau} \assign \partial_t \n$ and
$\n \in C^{\infty}_{\sharp} ([- \pi, \pi], \RR^2)$. More explicitly, we have
$\tmmathbf{\tau} (t) = (- \sin t, \cos t)$ and $\n (t) = (\cos t, \sin t)$.
Then, we set
\begin{equation}
  \uu_{\bot} = u_1 \tmmathbf{\tau}+ u_2 \n \label{eq:decompuperpj}
\end{equation}
with $u_1 \assign \uu_{\bot} \cdot \tmmathbf{\tau}$ and $u_2 \assign
\uu_{\bot} \cdot \n$. Note that both $u_1$ and $u_2$ belong to $H^1_{\sharp}
([- \pi, \pi], \RR)$.

  In this moving frame, the energy \eqref{eq:enFuperp} assumes the expression
\begin{equation}
  \GG (\uu_{\bot}) = \frac{1}{2 \pi}  \int_{- \pi}^{\pi} (\partial_t u_1 (t) +
  u_2 (t))^2 + (\partial_t u_2 (t) - u_1 (t))^2 + \kappa^2 u_1^2 (t) \mathd t
  \label{eq:FtoFourier}
\end{equation}
and the constraint \eqref{eq:constrR2} reads as
\begin{equation}
  \frac{1}{2 \pi}  \int_{- \pi}^{\pi} u_1^2 (t) + u_2^2 (t) \mathd t = 1.
  \label{eq:constforFour}
\end{equation}
In what follows, we denote by $\dot{\ell}_2 (\ZZ, \CC^2)$ the Sobolev space of
square summable sequences $(\tmmathbf{c}_n)_{n \in \ZZ} \assign (c_{1, n},
c_{2, n})_{n \in \ZZ}$ in $\ell_2 (\ZZ, \CC^2)$ such that
$(n\tmmathbf{c}_n)_{n \in \ZZ}$ is still in $\ell_2 (\ZZ, \CC^2)$. Every
in-plane vector field $\uu_{\bot} \in H^1_{\sharp} ([- \pi, \pi], \RR^2)$,
with components $u_1 \assign \uu_{\bot} \cdot \tmmathbf{\tau}$ and $u_2
\assign \uu_{\bot} \cdot \n$, can then be represented in Fourier series as
follows
\begin{equation}
  u_1 (t) = \sum_{n \in \ZZ} c_{1, n} e^{int}, \qquad u_2 (t) = \sum_{n \in
  \ZZ} c_{2, n} e^{int}, \label{eq:expruFourier}
\end{equation}
for some $(\tmmathbf{c}_n)_{n \in \ZZ} \assign (c_{1, n}, c_{2, n})_{n \in
\ZZ} \in \dot{\ell}_2 (\ZZ, \CC^2)$ and with $i$ the imaginary unit. As a side
remark, we note that
\begin{equation}
  \langle \uu_{\bot} \rangle = \langle u_1 \tmmathbf{\tau} \rangle + \langle
  u_2 \n \rangle = \sum_{n \in \{\pm 1\}} \langle c_{1, n} e^{int}
  \tmmathbf{\tau} \rangle + \sum_{n \in \{\pm 1\}} \langle c_{2, n} e^{int} \n
  \rangle .
\end{equation}
Therefore, the information on the average of $\uu_{\bot}$ is contained in the
harmonics of order $|n| = 1$.

Next, we represent the energy functional $\GG$ given by \eqref{eq:FtoFourier},
in the Fourier domain. For that, we use Parseval's theorem to restate the
energy $\GG$ in \eqref{eq:FtoFourier} in terms of Fourier coefficients as
follows
\[ \GG (\uu_{\bot}) = \sum_{n \in \ZZ} (n^2 + 1) (| c_{1, n} |^2 + | c_{2, n}
   |^2) + \kappa^2 | c_{1, n} |^2 - 4 n \Im [c_{1, n} \overline{c_{2, n}}] .
\]
Also, we take advantage of the symmetry properties of the complex Fourier
coefficients. Indeed, since we are representing real-valued functions, we know
that for every $n \in \NN$ there holds
\begin{equation}
  c_{1, 0}, c_{2, 0} \in \RR, \quad c_{1, - n} = \overline{c_{1, n}}, \quad
  c_{2, - n} = \overline{c_{2, n}} . \label{eq:symFourCoeff}
\end{equation}
After that, the energy $\GG (\uu_{\bot})$ can be represented under the form
\begin{align}
  \GG (\uu_{\bot}) & =  (| c_{1, 0} |^2 + | c_{2, 0} |^2) + \kappa^2 | c_{1,
  0} |^2 \nonumber\\
  &  \quad + 2 \sum_{n \geqslant 1} (n^2 + 1) (| c_{1, n} |^2 + | c_{2, n} |^2) +
  \kappa^2 | c_{1, n} |^2 - 4 n \Im [c_{1, n} \overline{c_{2, n}}] . 
\end{align}
We would like to interpret the previous expression as an energy functional on
$\dot{\ell}_2 (\ZZ_{\geq 0}, \RR^2 \times \RR^2)$. For that, we first observe that $\Im
[c_{1, n} \overline{c_{2, n}}] = J (\Re [c_{2, n}], \Im [c_{2, n}]) \cdot (\Re
[c_{1, n}], \Im [c_{1, n}])$ with the antisymmetric matrix
\begin{equation}
  J \assign \left(\begin{array}{cc}
    0 & - 1\\
    1 & 0
  \end{array}\right)
\end{equation}
representing the unitary image in matrix form, and then we set
\begin{equation}
  \tmmathbf{c}_{1, 0} = \left(\begin{array}{c}
    c_{1, 0}\\
    0
  \end{array}\right), \tmmathbf{c}_{2, 0} = \left(\begin{array}{c}
    c_{2, 0}\\
    0
  \end{array}\right) \label{eq:2simplify1},
\end{equation}
\begin{equation}
  \tmmathbf{c}_{1, n} = \sqrt{2} \left(\begin{array}{c}
    \Re [c_{1, n}]\\
    \Im [c_{1, n}]
  \end{array}\right), \tmmathbf{c}_{2, n} = \sqrt{2} \left(\begin{array}{c}
    \Re [c_{2, n}]\\
    \Im [c_{2, n}]
  \end{array}\right) \quad n \geqslant 1. \label{eq:2simplify2}
\end{equation}
In this way, the minimization problem
\eqref{eq:FtoFourier}-\eqref{eq:constforFour} for $\GG$ is equivalent to the
minimization (in the space of sequences $\dot{\ell}_2 (\ZZ_{\geq 0}, \RR^2 \times \RR^2)$) of the energy
functional
\begin{equation}
  \GG (\uu_{\bot}) = \sum_{n \geqslant 0} (n^2 + 1) (| \tmmathbf{c}_{1, n} |^2
  + | \tmmathbf{c}_{2, n} |^2) + \kappa^2 | \tmmathbf{c}_{1, n} |^2 - 4 n 
  J\tmmathbf{c}_{2, n} \cdot \tmmathbf{c}_{1, n} \label{eq:enc1c2new}
\end{equation}
with $\tmmathbf{c}_{1, 0} \cdot \tmmathbf{e}_2 =\tmmathbf{c}_{2, 0} \cdot
\tmmathbf{e}_2 = 0$, and subject to the constraint
\begin{equation}
  \sum_{n \geqslant 0} | \tmmathbf{c}_{1, n} |^2 + | \tmmathbf{c}_{2, n} |^2 =
  1. \label{eq:enc1c2newconstr}
\end{equation}
If $(\tmmathbf{c}_{1, n}, \tmmathbf{c}_{2, n})$ minimizes the energy
\eqref{eq:enc1c2new}-\eqref{eq:enc1c2newconstr}, then the Euler-Lagrange
equations associated with the minimization problem gives the existence of a
Lagrange multiplier $\lambda \in \RR$ such that for every $n \geqslant 0$
there holds
\begin{align}
  (n^2 + 1) \tmmathbf{c}_{1, n} + \kappa^2 \tmmathbf{c}_{1, n} - 2 n 
  J\tmmathbf{c}_{2, n} & =  \lambda \tmmathbf{c}_{1, n}, \\
  (n^2 + 1) \tmmathbf{c}_{2, n} + 2 n  J\tmmathbf{c}_{1, n} & = \lambda
  \tmmathbf{c}_{2, n} . 
\end{align}
from which one easily gets that $\lambda = \GG (\uu_{\bot})$. Therefore the
previous two relations can be rewritten under the form
\begin{align}
  \left( n^2 + 1 + \kappa^2 - \GG (\uu_{\bot}) \right) \tmmathbf{c}_{1, n} & =
   2 n  J\tmmathbf{c}_{2, n}  \label{eq:takemod1}\\
  \left( n^2 + 1 - \GG (\uu_{\bot}) \right) \tmmathbf{c}_{2, n} & = - 2 n 
  J\tmmathbf{c}_{1, n} .  \label{eq:takemod2}
\end{align}
Note that for $n = 0$, the first of the previous two relations gives $\left( 1
+ \kappa^2 - \GG (\uu_{\bot}) \right) c_{1, 0} = 0$ and, therefore, by
\eqref{eq:ubene} we have
\begin{equation}
  c_{1, 0} = 0. \label{eq:c1neq0}
\end{equation}
Also, observe that \eqref{eq:ubene} assures that for every $n \geqslant 0$
there holds
\[ \left( n^2 + 1 + \kappa^2 - \GG (\uu_{\bot}) \right) > 0, \qquad \left( n^2
   + 1 - \GG (\uu_{\bot}) \right) \geqslant 0, \]
with the second inequality being strict for every $n \geqslant 1$. Therefore, from
\eqref{eq:takemod1} and \eqref{eq:takemod2} we infer that for every $n
\geqslant 1$ there holds
\begin{align}
  \tmmathbf{c}_{1, n} & \overset{\eqref{eq:takemod1}}{=}  \frac{2 n}{\left(
  n^2 + 1 + \kappa^2 - \GG (\uu_{\bot}) \right)}  J\tmmathbf{c}_{2, n}
  \nonumber\\
  & \overset{\eqref{eq:takemod2}}{=}  \frac{4 n^2}{\left( n^2 + 1 - \GG
  (\uu_{\bot}) \right) \left( n^2 + 1 + \kappa^2 - \GG (\uu_{\bot}) \right)} 
  \tmmathbf{c}_{1, n} . \nonumber
\end{align}
Hence, if $\tmmathbf{c}_{1, n} \neq 0$ then necessarily $\left( n^2 + 1 - \GG
(\uu_{\bot}) \right) \left( n^2 + 1 + \kappa^2 - \GG (\uu_{\bot}) \right) = 4
n^2$ which, given the upper bound $\GG (\uu_{\bot}) \leqslant \kappa^2 / 2$ in
\eqref{eq:ubene}, admits the only possible solution
\[ \GG (\uu_{\bot}) = \frac{\kappa^2}{2} + n^2 + 1 - \frac{1}{2}
   \sqrt{\kappa^4 + 16 n^2} . \]
Also, the upper bound $\GG (\uu_{\bot}) \leqslant 1$ in \eqref{eq:ubene} gives
that $1 \leqslant n^2 \leqslant 4 - \kappa^2$. Therefore we deduce that we can
have $\tmmathbf{c}_{1, n} \neq 0$ only when $\kappa^2 \leqslant 3$ and, in
this latter case (i.e., if $\tmmathbf{c}_{1, n} \neq 0$ and $0 \neq \kappa^2
\leqslant 3$), one necessarily has $n^2 < 4$, i.e., $n \leqslant 1$.

Overall, we get that it is always the case (i.e., for every $\kappa^2
> 0$) that
\[ c_{1, 0} = 0 \quad \text{and} \quad \tmmathbf{c}_{1, n} =\tmmathbf{c}_{2,
   n} = 0 \quad \forall n \geqslant 2. \]
Moreover, if $\kappa^2 > 3$ then the last relation strengthen to
$\tmmathbf{c}_{1, n} =\tmmathbf{c}_{2, n} = 0$ for every $n \geqslant 1$. In
other words, the minimization problem
\eqref{eq:enc1c2new}-\eqref{eq:enc1c2newconstr} reduces to
a finite-dimensional minimization problem and it is convenient to investigate
the regimes $\kappa^2 > 3$, and $\kappa^2 \leqslant 3$ separately.

{\noindent}{\tmstrong{The regime $\kappa^2 > 3$.}} The problem
\eqref{eq:enc1c2new}-\eqref{eq:enc1c2newconstr} trivializes to
\begin{equation}
  \text{minimize } c_{2, 0}^2
\end{equation}
subject to $c_{2, 0}^2 = 1$. This immediately gives $u_1 (t) = 0$ and $u_2 (t)
\equiv \pm 1$ as minimizer of the energy \eqref{eq:FtoFourier} and, therefore
({\tmabbr{cf.}}~\eqref{eq:decompuperpj}), $\uu_{\bot} = \pm \tmmathbf{n}$ are
the only two global minimizers of the energy when $\kappa^2 > 3$. This proves
statement i. in Theorem~\ref{thm:PoincIneq}.

{\noindent}{\tmstrong{The regime $\kappa^2 \leqslant 3$.}} The problem
\eqref{eq:enc1c2new}-\eqref{eq:enc1c2newconstr} becomes
\begin{equation}
  \text{minimize } c_{2, 0}^2 + 2 (| \tmmathbf{c}_{1, 1} |^2 + |
  \tmmathbf{c}_{2, 1} |^2) + \kappa^2 | \tmmathbf{c}_{1, 1} |^2 - 4 
  J\tmmathbf{c}_{2, 1} \cdot \tmmathbf{c}_{1, 1} \label{eq:constrnewj0}
\end{equation}
subject to
\begin{equation}
  c_{2, 0}^2 + | \tmmathbf{c}_{1, 1} |^2 + | \tmmathbf{c}_{2, 1} |^2 = 1.
  \label{eq:constrnewj}
\end{equation}
We can incorporate part of the constraint into the energy functional and solve
for
\begin{equation}
  \text{minimize } 1 + | \tmmathbf{c}_{1, 1} |^2 + | \tmmathbf{c}_{2, 1} |^2 +
  \kappa^2 | \tmmathbf{c}_{1, 1} |^2 - 4  J\tmmathbf{c}_{2, 1} \cdot
  \tmmathbf{c}_{1, 1}  \label{eq:newminrelprob1}
\end{equation}
subject to the relaxed
\begin{equation}
  | \tmmathbf{c}_{1, 1} |^2 + | \tmmathbf{c}_{2, 1} |^2 \leqslant 1.
  \label{eq:newminrelprob2}
\end{equation}
In fact, if $\tmmathbf{c}_{1, 1}, \tmmathbf{c}_{2, 1} \in \RR^2$ solve
\eqref{eq:newminrelprob1}-\eqref{eq:newminrelprob2}, it is then sufficient to
set $c_{2, 0}^2 = 1 - | \tmmathbf{c}_{1, 1} |^2 + | \tmmathbf{c}_{2, 1} |^2$
to get a solution of \eqref{eq:constrnewj0}-\eqref{eq:constrnewj}. To solve
\eqref{eq:newminrelprob1}-\eqref{eq:newminrelprob2} we use polar coordinates
in the plane. We set
\begin{equation}
  \tmmathbf{c}_{1, 1} = \rho_1 (\cos \theta_1, \sin \theta_1), \quad
  \tmmathbf{c}_{2, 1} = \rho_2 (\cos \theta_2, \sin \theta_2),
  \label{eq:polarcoordsplanec1c2}
\end{equation}
We note that $J\tmmathbf{c}_{2, 1} \cdot \tmmathbf{c}_{1, 1} = \rho_1 \rho_2
\sin (\theta_1 - \theta_2)$ and we reformulate the minimization problem as
\begin{equation}
  \text{minimize } 1 + \rho_1^2 + \rho_2^2 + \kappa^2 \rho_1^2 - 4  \rho_1
  \rho_2 \sin (\theta_1 - \theta_2)
\end{equation}
subject to
\begin{equation}
  \rho_1, \rho_2 \geqslant 0 \quad \text{and} \quad \rho_1^2 + \rho_2^2
  \leqslant 1. \label{eq:constrnew1}
\end{equation}
We note that $\theta_1, \theta_2$ do not play any role in the constraint
\eqref{eq:constrnew1}, therefore at a minimum point one must have
\begin{equation}
  \theta_1 - \theta_2 = \pi / 2 \quad \text{mod } 2 \pi . \label{eq:th1th2mod}
\end{equation}
It remains to minimize
\begin{equation}
  \text{minimize } g (\rho_1, \rho_2) \assign 1 + \rho_1^2 + \rho_2^2 +
  \kappa^2 \rho_1^2 - 4  \rho_1 \rho_2 \label{eq:grho1rho2}
\end{equation}
subject to the constraint $\rho_1, \rho_2 \geqslant 0$ and $\rho_1^2 +
\rho_2^2 \leqslant 1$.

{\noindent}{\emph{The regime $\kappa^2 = 3$.}} One can easily check that when
$\kappa^2 = 3$ the function $g$ reduces to $g (\rho_1, \rho_2) = 1 + (2 \rho_1
- \rho_2)^2$. Therefore the energy is minimized when $\rho_2 = 2 \rho_1$ and
$\rho_1^2 \leqslant 1 / 5$. But then, from \eqref{eq:polarcoordsplanec1c2} and
\eqref{eq:th1th2mod} we infer that
\begin{equation}
  \tmmathbf{c}_{1, 1} = \rho_1 (\cos \theta_1, \sin \theta_1), \quad
  \tmmathbf{c}_{2, 1} = 2 \rho_1 (\cos \theta_2, \sin \theta_2), \quad
  \theta_1 - \theta_2 = \pi / 2,
\end{equation}
which in turn, given \eqref{eq:2simplify1} and \eqref{eq:2simplify2},
translate into
\begin{equation}
  c_{2, 0}^2 = 1 - 5 \rho_1^2, \quad c_{1, 1} = \frac{\sqrt{2}}{2} \rho_1 e^{i
  \theta}, \quad c_{2, 1} = \sqrt{2} \rho_1 e^{i (\theta - \pi / 2)}
\end{equation}
for an arbitrary $\theta \in \RR$. It follows that $u_1 (t) = c_{1, 1} e^{it}
+ \overline{c_{1, 1}} e^{- it} = \sqrt{2} \rho_1 \cos (\theta + t)$ and $u_2
(t) = c_{2, 0} + c_{2, 1} e^{it} + \overline{c_{2, 1}} e^{- it} = \pm \sqrt{1
- 5 \rho_1^2} + 2 \sqrt{2} \rho_1 \sin (\theta + t)$. The corresponding family
of minimizers reads as ({\tmabbr{cf.}}~\eqref{eq:decompuperpj})
\[ \uu_{\bot} (t) = \sqrt{2} \rho_1 \cos (\theta + t) \tmmathbf{\tau} (t) +
   \left[ \pm \sqrt{1 - 5 \rho_1^2} + 2 \sqrt{2} \rho_1 \sin (\theta + t)
   \right] \n (t) \]
for arbitrary $\theta \in [- \pi, \pi]$ and $0 \leqslant \rho_1 \leqslant 1 /
\sqrt{5}$. The common value of the energy is $\GG (\uu_{\bot}) = 1$. This
proves statement ii. in Theorem~\ref{thm:PoincIneq}.

{\noindent}{\emph{The regime $\kappa^2 < 3$.}} To study the minimization
problem \eqref{eq:grho1rho2} in the regime $\kappa^2 < 3$ we rely on polar
coordinates in the plane. We set $(\rho_1, \rho_2) = r \sigma$ with $0
\leqslant r \leqslant 1$, $\sigma = (\cos \phi, \sin \phi) \in \Stwo^1$, $\phi
\in [0, \pi / 2]$ and minimize in $\phi$ the function
\begin{equation}
  g (r, \phi) \assign 1 + r^2 (1 + \kappa^2 \cos^2 \phi - 2 \sin 2 \phi) .
\end{equation}
For that we note that when $\kappa^2 < 3$ the quantity $1 + \kappa^2 \cos^2
\phi - 2 \sin 2 \phi$ is strictly negative for a range of angles included in
$[0, \pi / 2]$. Indeed, $1 + \kappa^2 \cos^2 \phi - 2 \sin 2 \phi$ is strictly
negative when
\[ \arctan \left( 2 - \sqrt{3 - \kappa^2} \right) < \phi < \arctan \left( 2 +
   \sqrt{3 - \kappa^2} \right) \]
that is, when $\phi \in [\pi / 12, 5 \pi / 12] \subseteq [0, \pi / 2]$.
Therefore, when $\kappa^2 < 3$, in order to minimize $g (r, \phi)$ we want to
take maximal $r^2 = 1$. The energy minimization problem then further reduces
to a parametric problem in the single variable $\phi$:
\begin{equation}
  g (\phi) = \kappa^2 \cos^2 \phi + 2 - 2 \sin (2 \phi) . \label{eq:gsimpl6}
\end{equation}
For what follows, it is convenient to set $\omega_{\kappa}^2 \assign
\sqrt{\kappa^4 + 16}$. The first-order minimality conditions can be written
under the form
\begin{equation}
  \frac{\kappa^2}{\omega_{\kappa}^2} \sin (2 \phi) +
  \frac{4}{\omega_{\kappa}^2} \cos (2 \phi) = 0.  \label{eq:fomincond}
\end{equation}
It is then clear that for every $\kappa^2 > 0$ there exists a unique angle in
$\phi_{\kappa} \in [0, \pi / 2]$ such that
\begin{equation}
  \cos (2 \phi_{\kappa}) = \kappa^2 / \omega_{\kappa}^2, \qquad \sin (2
  \phi_{\kappa}) = 4 / \omega_{\kappa}^2 . \label{eq:eqsforphik}
\end{equation}
The observation allows us to rewrite the first-order minimality condition
\eqref{eq:fomincond} under the form $\sin (2 (\phi + \phi_{\kappa})) = 0$.
Thus, given $0 < \kappa^2 < 3$, once computed $\phi_{\kappa}$, the minimal
energy is achieved at $\phi = - \phi_{\kappa} + \pi / 2$ and $r = 1$ (recall
that $\phi, \phi_{\kappa} \in [0, \pi / 2]$). This leads to $\rho_1 = \cos (-
\phi_{\kappa} + \pi / 2) = \sin \phi_{\kappa}$, $\rho_2 = \sin (-
\phi_{\kappa} + \pi / 2) = \cos \phi_{\kappa}$ with $\phi_{\kappa} \in [0, \pi
/ 2]$ being the unique solution of \eqref{eq:eqsforphik}. But this means that
the minimal coefficients $c_{2, 0}$, $\tmmathbf{c}_{1, 1}$, and
$\tmmathbf{c}_{2, 1}$ are given by
({\tmabbr{cf.}}~\eqref{eq:polarcoordsplanec1c2})
\[ \tmmathbf{c}_{1, 1} = (\sin \phi_{\kappa}) \sigma, \quad \tmmathbf{c}_{2,
   1} = - (\cos \phi_{\kappa}) J \sigma, \quad c_{2, 0} = 0 \]
for arbitrary $\sigma \in \Stwo^1$. The corresponding minimal energy is given
by (cf.~\eqref{eq:eqsforphik})
\begin{align}
  2 (| \tmmathbf{c}_{1, 1} |^2 + | \tmmathbf{c}_{2, 1} |^2) + \kappa^2 |
  \tmmathbf{c}_{1, 1} |^2 - 4  J\tmmathbf{c}_{2, 1} \cdot \tmmathbf{c}_{1, 1}
  & =  2 + \frac{1}{2} \kappa^2  (1 - \cos (2 \phi_{\kappa})) - 2 \sin (2
  \phi_{\kappa}) \nonumber\\
  & =  \frac{1}{2}  (\kappa^2 -
  \omega_{\kappa}^2 + 4) . \nonumber
\end{align}
Going back to the original configuration we have
({\tmabbr{cf.}}~\eqref{eq:2simplify1}-\eqref{eq:2simplify2})
\[ c_{1, 1} = \frac{\sqrt{2}}{2} (\sin \phi_{\kappa}) e^{i \theta}, \quad
   c_{2, 1} = - i \frac{\sqrt{2}}{2} (\cos \phi_{\kappa}) e^{i \theta} \]
for an arbitrary $\theta \in \RR$ and, therefore, $u_1 (t) = c_{1, 1} e^{it} +
\overline{c_{1, 1}} e^{- it} = \sqrt{2} (\sin \phi_{\kappa}) \cos (\theta +
t)$ and $u_2 (t) = c_{2, 1} e^{it} + \overline{c_{2, 1}} e^{- it} = \sqrt{2}
(\cos \phi_{\kappa}) \sin (\theta + t)$. The corresponding family of
minimizers reads as ({\tmabbr{cf.}}~\eqref{eq:decompuperpj})
\[ \uu_{\bot} (t) = \sqrt{2} (\sin \phi_{\kappa}) \cos (\theta + t)
   \tmmathbf{\tau} (t) + \sqrt{2} (\cos \phi_{\kappa}) \sin (\theta + t) \n
   (t) \]
with $\phi_{\kappa} \in [0, \pi / 2]$ given by $\phi_{\kappa} \assign
\frac{1}{2} \arctan (4 / \kappa^2)$ and $\theta \in [- \pi, \pi]$ arbitary.
The common value of the energy is $\GG (\uu_{\bot}) = \frac{1}{2}  (\kappa^2 -
\omega_{\kappa}^2 + 4)$. This proves statement iii. in
Theorem~\ref{thm:PoincIneq}.

Finally, we note that for $\kappa^2 < 3$
\[ | \uu_{\bot} (t) |^2 = 2 \sin^2 (\phi_{\kappa}) \cos^2 (t + \theta) + 2
   \cos^2 (\phi_{\kappa}) \sin^2 (t + \theta) \]
and, therefore, $\uu_{\bot}$ is $\Stwo^2$ valued only when $\cos (2
\phi_{\kappa}) = 0$. But this never happens because from \eqref{eq:eqsforphik}
we know that $\cos (2 \phi_{\kappa}) = \kappa^2 / \omega_{\kappa}^2$.
Therefore, when $0 < \kappa^2 < 3$, there are no minimizers of $\GG$ in
problem \eqref{eq:FscriptforPoinc}-\eqref{eq:contru} that are
$\Stwo^2$-valued. This gives \eqref{eq:eqsforphikthm} and concludes the proof
of the theorem.
\end{proof}

\section{The stability of in-plane configurations}\label{sec:numericalinplane}

This section is devoted to the analysis of in-plane minimizers of the energy
\eqref{eq:micromagenfunonCccoordgen}. The interest in such configurations is
motivated by numerical simulations. Indeed, numerical schemes for the analysis
of ground states of $\mathcal{F}$ seem to converge towards solutions that are
in-plane. The phenomenon, enforced by Theorem \ref{thm:BFTCylinder} when
$\kappa^2 \geqslant 3$, and partially endorsed by Lemma~\ref{lemma:e3eq0} for
$\kappa^2 < 3$, motivates the following conjecture.

{\noindent}{\tmstrong{Conjecture $(C)$. }}\emph{For every $\kappa^2 > 0$
the minimizers in $H^1_{\sharp} ([- \pi, \pi], \Stwo^2)$ the energy functional
{\opt}{\tmabbr{{\emph{cf}}.}}~{\emph{\eqref{eq:reden}}}{\cpt}
\begin{equation}
  \mathcal{F} (\uu) \assign \int_{- \pi}^{\pi} | \partial_t{\uu}
  (t) |^2 \mathd t \; + \; \kappa^2 \int_{- \pi}^{\pi} | \uu (t)
  \times \n (t) |^2 \mathd t \label{eq:redennewloccoord}
\end{equation}
are in-plane. In other words, if $\uu \in H^1_{\sharp} ([- \pi, \pi],
\Stwo^2)$ minimizes {\emph{\eqref{eq:redennewloccoord}}} then $\uu
\cdot \tmmathbf{e}_3 \equiv 0$ in $[- \pi, \pi]$.}{\medskip}

{\noindent}Theorem \ref{thm:BFTCylinder} assures that conjecture $(C)$ is true
when $\kappa^2 \geqslant 3$, as $\pm \n$ are the only global
minimizers of $\mathcal{F}$. When $\kappa^2 < 3$, the answer remains open.
Indeed, while it is simple to prove that all minimizers are $\Stwo^1$-valued
when, as in Lemma~\ref{lemma:e3eq0}, the $\Stwo^2$-valued constraint is
relaxed to the energy constraint
\begin{equation}
  \| \uu \|^2_{L^2_{\sharp} ( [- \pi, \pi], \RR^3 )} = 2
  \pi,
\end{equation}
the situation seems more involved for $\Stwo^2$-valued configurations.

Let us comment a little bit more about some common aspects of the numerical
schemes available to compute energy-minimizing maps. We focus on the iteration
scheme introduced by Alouges in {\cite{MR1472192}} for computing stable
$\Stwo^2$-valued harmonic maps on bounded domains of $\RR^3$, but our
observations transfer to other numerical schemes, e.g., the dissipative flow
governed by the Landau--Lifshitz--Gilbert equation (see, e.g.,
{\cite{alouges1992global,Alouges_2008,di2017linear}}). The algorithm proposed
in {\cite{MR1472192}} has the advantage of operating at a continuous level;
this allows us to use it as a versatile theoretical tool to obtain the
existence of solutions with prescribed properties. Starting from an initial
guess $\m_0 \in H^1_{\sharp} ([- \pi, \pi], \Stwo^2)$, the scheme
builds a sequence $(\m_j)_{j \in \NN}$ of energy-decreasing
$\Stwo^2$-valued configurations which eventually converges, weakly in
$H^1_{\sharp} ([- \pi, \pi], \Stwo^2)$, to a critical point $\m_{\infty}$ of
the energy $\mathcal{F}$. The algorithm preserves specific structural
properties of the initial guess $\m_0$. While structure-preserving
features are most often a strength of numerical schemes, other times they
represent their biggest weakness. For example, as we are going to show, the
algorithm retains the following properties of the initial guess
{\cite{DiFratta2021}}:
\begin{enumerate}
  \item[\emph{i.}] If the initial guess $\m_0$ is axially symmetric
  (with respect to the $z$-axis), so are the elements of the sequence
  $(\m_j)_{j \in \NN}$ produced by the iterative scheme and the weak
  limit $\m_{\infty}$.
  
  \item[\emph{ii.}] If the initial guess $\m_0$ is in-plane, then all
  the elements of the sequence $(\m_j)_{j \in \NN}$ produced by the
  iterative scheme are in-plane, as well as the weak limit
  $\m_{\infty}$. Moreover, if the initial guess $\m_0$ is
  in-plane and in a prescribed homotopy class, so is the weak limit
  $\m_{\infty}$ of $(\m_j)_{j \in \NN}$.
\end{enumerate}
Point {\emph{ii}} tells us that regardless of whether conjecture $(C)$ is true
or false, in-plane configurations appear in simulations and therefore are of
interest. For this reason, the second half of the section focuses on the
characterization of the in-plane critical points of the energy functional
$\mathcal{F}$ (see \eqref{eq:redennewloccoord}) and the analysis of their
minimality properties.

In order to state the main result of this section, we need to introduce some
notation. In what follows, as before, we denote by $\n (t) = (\cos
t, \sin t, 0)$ and $\tmmathbf{\tau} (t) \assign \partial_t \n (t)$ the normal
and tangential fields to $\Stwo^1 \times \{ 0 \}$. Also, for any $\kappa^2 >
0$ we denote by $\alpha_{\kappa} > 0$ the {\emph{unique}} solution of the
equation
\begin{equation}
  \frac{1}{2 \pi} \int_{- \pi}^{\pi} \frac{1}{\sqrt{\alpha^2_{\kappa} +
  \kappa^2 \sin^2 x}} \mathd x = 1. \label{eq:uniqsolalpha}
\end{equation}
The uniqueness of the solution of \eqref{eq:uniqsolalpha} comes from the fact
that for every $\kappa^2 > 0$ the continuous function
\begin{equation}
  \alpha \in \RR_+ \mapsto \frac{1}{2 \pi} \int_{- \pi}^{\pi}
  \frac{1}{\sqrt{\alpha^2 + \kappa^2 \sin^2 x}} \mathd x
\end{equation}
is decreasing in $\alpha$, diverges to $+ \infty$ when $\alpha \rightarrow 0$,
and converges to $0$ when $\alpha \rightarrow + \infty$.

After that, given $\alpha_{\kappa}$, we denote by $F_{\kappa}$ the elliptic
integral of the first kind defined for any $\theta \in \RR$ by
\begin{equation}
  F_{\kappa} (\theta) \assign \int_{- \pi}^{\theta} \frac{1}{\sqrt{1 +
  (\kappa^2 / \alpha^2_{\kappa}) \sin^2 x}} \mathd x.  \label{eq:EllipticErev}
\end{equation}
and  we denote its inverse function, which is usually referred to as the
Jacobi amplitude function, as $\tmop{am}_{\kappa} \assign F^{- 1}_{\kappa}$.
Finally, we define $E_\kappa$ to be the complete elliptic integral
\begin{equation}
  E_{\kappa} \assign \int_{- \pi}^{\pi}  \sqrt{1 + (\kappa^2 /
  \alpha^2_{\kappa}) \sin^2 x} \mathd x. \label{eq:EllipticE}
\end{equation}
It follows from \eqref{eq:uniqsolalpha} that $F_{\kappa} (\theta + 2 \pi) = F_{\kappa} (\theta) + 2 \pi \alpha_{\kappa}$.
Hence, the inverse function $\tmop{am}_{\kappa}$ satisfies the identity $\tmop{am}_{\kappa} (y + 2
\pi \alpha_{\kappa}) = \tmop{am}_{\kappa} (y) + 2 \pi$ which, in particular, assures that the minimizers identified in the next Theorem~\ref{thm:mainS1S1} (cf.~\eqref{eq:mminmdegzero}-\eqref{eq:minmdegzero}) are $2 \pi$-periodic.
\begin{theorem}[\tmname{in-plane minimizers}]
  \label{thm:mainS1S1}Let $\kappa^2 > 0$. {\emph{\begin{figure}[t]
    {\includegraphics[width=\textwidth]{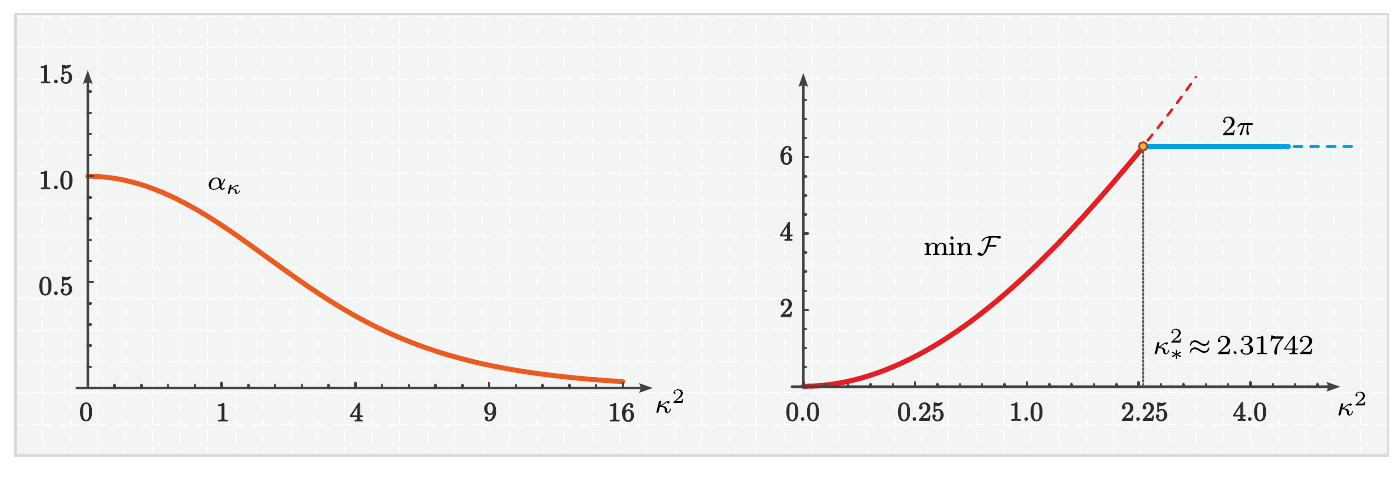}}
    \caption{\label{Fig:optimalEn}There is a critical value $\kappa_{\ast}^2$
    of the anisotropy parameter, $\kappa_{\ast}^2 \approx 2.31742$, below
    which the global minimizers of \eqref{eq:redennewloccoord} have degree
    zero, and above which the only two global minimizers are the normal vector
    fields $\pm \n$ (and have degree one).
    }
  \end{figure}}}If $\m_{\bot}$ is a {\opt}global{\cpt} minimizer in
  $H^1_{\sharp} ([- \pi, \pi], \Stwo^1)$ of the energy functional
  $\mathcal{F}$
  {\opt}{\tmabbr{{\emph{cf}}.}}~{\emph{\eqref{eq:redennewloccoord}}}{\cpt},
  then either $\jdeg  \m_{\bot} = 0$ or $\jdeg  \m_{\bot} = 1$. Precisely,
  there exists a threshold value $\kappa^2_{\ast}$ of the anisotropy parameter
  such that the following dichotomy holds:
  \begin{enumerate}
    \item[i.] If $\kappa^2 > \kappa_{\ast}^2$, then any global minimizer has
    degree one. Moreover, for every $\kappa^2 > 0$, the normal fields $\pm
    \n$ are the only two global minimizers of $\mathcal{F}$ in the
    homotopy class $\left\{ \jdeg  \m_{\bot} = 1 \right\}$.
    
    \item[ii.] If $\kappa^2 < \kappa_{\ast}^2$, then any global minimizer has
    degree zero. Also, for any $\kappa^2 > 0$, the minimizers of
    $\mathcal{F}$ in the homotopy class $\left\{ \jdeg  \m_{\bot} \right\} =
    0$ are given by
    {\opt}{\emph{{\tmabbr{cf.}}~Figure}}~{\emph{\ref{fig:kdegzero}}}{\cpt}
    \begin{equation} \label{eq:mminmdegzero}
      \m_{\bot} (t) = \sin \theta (t) \tmmathbf{\tau} (t) + \cos \theta (t) \n
      (t)
    \end{equation} 
    with $\theta$ \opt a strictly decreasing function\cpt\ being any element of the family
    \begin{equation}
      \theta (t) = \tmop{am}_{\kappa} (- \alpha_\kappa t + b), \quad b \in \RR\, ,
      \label{eq:minmdegzero}
    \end{equation}
    and $\alpha_\kappa>0$ the unique solution of \eqref{eq:uniqsolalpha}. Moreover, the minimal value of the energy is given by
    {\opt}{\emph{{\tmabbr{cf.}}~Figure}}~{\emph{\ref{Fig:optimalEn}}}{\cpt}
    \begin{equation}
      \mathcal{F} \left( \m_{\bot} \right) = - 2 \pi (1 + \alpha_{\kappa}^2) +
      2 \alpha_{\kappa} E_{\kappa} . \label{eq:minendegzero}
    \end{equation}
    \item[iii.] The exact value of $\kappa^2_{\ast}$ is determined as the solution
    of the equation
    \begin{equation}
      -2\pi\alpha_{\kappa}^2 + 2 \alpha_{\kappa} E_{\kappa} = 4 \pi,
    \end{equation}
    which gives $\kappa^2_{\ast} \approx 2.31742$.
    
    \item[iv.] For any $\kappa^2 > 0$, the degree-zero solutions
    {\emph{\eqref{eq:minmdegzero}}} are locally stable critical points of the
    energy $\mathcal{F}$ in {\emph{\eqref{eq:redennewloccoord}}}. Also, for
    any $\kappa^2 > 0$, degree-one solutions $\pm \n$ are local minimizers of
    the energy $\mathcal{F}$.
  \end{enumerate}
  Combining i and ii we get the following characterization of the energy
  landscape. The normal vector fields $\pm \n$ are the only two
  global minimizers of $\mathcal{F}$ when $\kappa^2 > \kappa_{\ast}^2$ and the
  common minimum value of the energy is $2 \pi$. When $\kappa^2 <
  \kappa_{\ast}^2$ the minimal energy depends on $\kappa$, it is given by
  {\emph{\eqref{eq:minendegzero}}}, and is reached when $\theta$ is given by
  {\emph{\eqref{eq:minmdegzero}}}. Finally, when $\kappa^2=\kappa_\ast^2$ both the degree one solutions $\pm\n$ and the degree zero solutions \eqref{eq:mminmdegzero}-\eqref{eq:minmdegzero} coexist as energy minimizers and the common value of the energy is $2\pi$~{\opt}{\emph{{\tmabbr{cf.}}~Figure}}~{\emph{\ref{Fig:optimalEn}}}{\cpt}.
\end{theorem}

\begin{remark}
  Note that, while in the $\Stwo^1$-valued setting the normal vector fields
  $\pm \n$ are local minimizers for every $\kappa^2 > 0$, this was
  not the case in the $\Stwo^2$-valued setting reported in
  Theorem~\ref{thm:BFTCylinder} (where $\pm \n$ become unstable for
  $\kappa^2 < 1$). The precise reason is that here we are restricted to the
  class of in-plane perturbations, whereas in the $\Stwo^2$-valued case the
  loss of stability is caused by perturbations in the $\e_3$ direction.
\end{remark}

\begin{proof}
  As in the proof of Theorem~\ref{thm:PoincIneq}, it is convenient to set
  $\m_{\bot} = m_1 \tmmathbf{\tau}+ m_2 \n$ with $m_1 \assign
  \m_{\bot} \cdot \tmmathbf{\tau}$ and $m_2 \assign
  \m_{\bot} \cdot \n$, where $\n (t) = (\cos t,
  \sin t, 0)$ and $\tmmathbf{\tau} (t) \assign \partial_t \n (t)$ are the
  normal and tangential fields to $\Stwo^1 \times \{ 0 \}$. Note that both
  $m_1$ and $m_2$ are in $H^1_{\sharp} ( [- \pi, \pi], \RR)$. In
  the moving frame $(\tmmathbf{\tau}, \n)$, the energy
  \eqref{eq:redennewloccoord} assumes the expression
  ({\tmabbr{cf.}}~\eqref{eq:FtoFourier})
  \begin{equation}
    \mathcal{F} (\m_{\bot}) = \int_{- \pi}^{\pi} | \partial_t m_2 -
    m_1 |^2 + | \partial_t m_1 + m_2 |^2 \mathd t + \kappa^2  \int_{-
    \pi}^{\pi} m_1^2  \hspace{0.17em} \mathd t. \label{eq:enmovframVal}
  \end{equation}
  Clearly, the vector field $(m_1, m_2)$ is $\Stwo^1$-valued. We lift it by
  setting
  \begin{equation}
    m_1 (t) \assign \sin \theta (t), \quad m_2 (t) \assign \cos \theta (t) . \label{eq:liftSob}
  \end{equation}
  After that, the energy \eqref{eq:enmovframVal} reads as
  \begin{align}
    \mathcal{F} (\m_{\bot}) =\; & \int_{- \pi}^{\pi} | \partial_t
    \theta (t) + 1 |^2 \mathd t + \kappa^2 \int_{- \pi}^{\pi} \sin^2 \theta
    (t) \mathd t  \label{eq:enmovframVal22}\\
    =\; & \int_{- \pi}^{\pi} | \partial_t \theta (t) |^2 + \kappa^2 \sin^2
    \theta (t) \mathd t + 2 \pi + 2 (\theta (\pi) - \theta (- \pi)) . 
  \end{align}
  \begin{figure}[t]
    {\includegraphics[width=\textwidth]{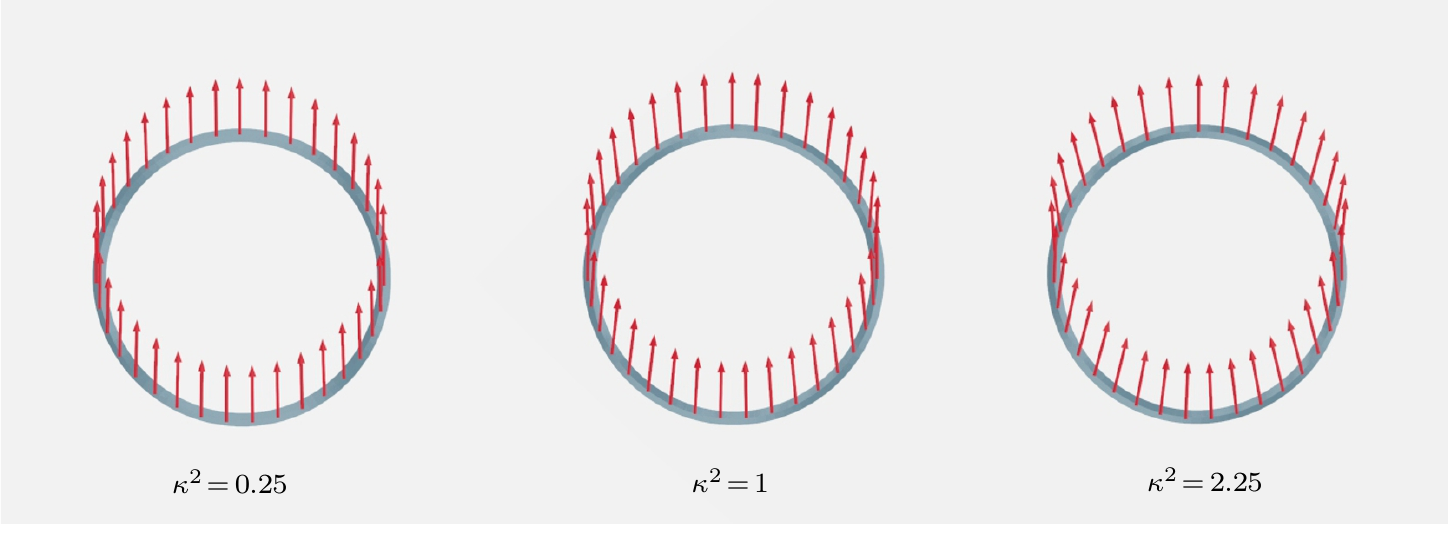}}
    \caption{\label{fig:kdegzero}A plot of the vector fields minimizing the
    energy \eqref{eq:redennewloccoord}. There is a critical value
    $\kappa_{\ast}^2$ of the anisotropy parameter, $\kappa_{\ast}^2 \approx
    2.31742$, below which the global minimizers of \eqref{eq:redennewloccoord}
    have degree zero, and above which the only two global minimizers are the
    normal vector fields $\pm \n$ (and have degree one). From left
    to right, we plot the minimizers for $\kappa^2 = 0.25$, $\kappa^2 = 1$,
    and $\kappa^2 = 2.25$.
    }
  \end{figure}It is clear that since $m_1, m_2 \in H^1_{\sharp} ( [- \pi,
  \pi], \RR)$ we necessarily have
  \begin{equation}
    \theta (\pi) - \theta (- \pi) = 2 \pi j \label{eq:defj}
  \end{equation}
  for some $j \in \ZZ$. Hence, the energy functional $\mathcal{F}$ takes the
  form
  \begin{equation}
    \mathcal{F} (\m_{\bot}) = \int_{- \pi}^{\pi} | \partial_t
    \theta (t) |^2 + \kappa^2 \sin^2 \theta (t) \mathd t + 2 \pi (1 + 2 j)
    \label{eq:valtemp1}
  \end{equation}
  The integer $j \in \ZZ$ is nothing but the degree of the $\Stwo^1$-valued
  map $(m_1, m_2)$ whose components are the coordinates of
  $\m_{\bot}$ in the moving frame $(\tmmathbf{\tau}, \n)$.
  Therefore, from \eqref{eq:defj} we get that
  \begin{equation}
    \jdeg \m_{\bot} = j + 1. \label{eq:degreemperp}
  \end{equation}
  In fact, the previous relation can also be obtained from \eqref{eq:defj} observing that $\m_{\bot}=(\cos (\theta (t) + t), \sin (\theta (t) + t), 0)$.

  The expression \eqref{eq:valtemp1} can be used to investigate the critical
  points of $\mathcal{F}$ in any prescribed homotopy class $j \in \ZZ$. Here,
  however, we are interested in global minimizers and, as we are going to
  show, this restricts the admissible homotopy classes to only two cases.
  
  First, we use \eqref{eq:defj} and Jensen's inequality for the Dirichlet part in \eqref{eq:valtemp1}
  to get that
  \begin{equation}
    \mathcal{F} (\m_{\bot})  \geqslant  2 \pi (1 + j )^2 +
    \kappa^2 \int_{- \pi}^{\pi} \sin^2 \theta (t) \mathd t 
    \label{eq:estvaltemp2}
  \end{equation}
  for every $\m_{\bot} \in H^1_{\sharp} ( \Stwo^1, \Stwo^1)$.
  Second, we recall that for every $\kappa^2 > 0$, we have $\mathcal{F}
  (\n) = 2 \pi$, as well as $\mathcal{F} (\tmmathbf{\sigma}) = \pi
  \kappa^2$ for any $\tmmathbf{\sigma} \in \Stwo^1$. Therefore, if
  $\m_{\bot}$ is a minimizer of \eqref{eq:valtemp1}, then
  necessarily
  \begin{equation}
    2 \pi (1 + j )^2 \leqslant \mathcal{F} (\m_{\bot}) \leqslant \min \{ 2
    \pi, \pi \kappa^2 \} . \label{eq:enboundsFS1}
  \end{equation}
  Since $2 \pi (1 + j )^2 > 2 \pi$ when $| 1 + j | > 1$, from the previous
  bounds, we get that the global minimizers of $\mathcal{F}$ in $H^1_{\sharp}
  ( \Stwo^1, \Stwo^1)$ have to satisfy the relation
  \eqref{eq:defj} with $j \in \{ - 2, - 1, 0 \}$. On the other hand, if $j = -
  2$ we get that $\mathcal{F} (\m_{\bot})$ is strictly greater than
  $2 \pi$ because the density $\kappa^2 (\sin \theta)^2$ gives a positive
  contribution when $j \neq 0$. Hence, necessarily $j \in \{ - 1, 0 \}$.
  Hence, recalling \eqref{eq:degreemperp},
  \[ \jdeg \m_{\bot} \in \{ 0, 1 \} . \]
  This proves the first part of the statement. 
  
Now, if $\m_{\bot}$ is a global minimizer of $\mathcal{F}$ and if $\jdeg
\m_{\bot} = 1$, i.e., if $j = 0$, then from \eqref{eq:enboundsFS1} we know
that $2 \pi \leqslant \mathcal{F} (\m_{\bot}) \leqslant \pi \kappa^2$. Hence,
necessarily $\kappa^2 \geqslant 2$ if $\jdeg \m_{\bot} = 1$. It follows that
$\jdeg \m_{\bot} = 0$ whenever $\kappa^2 < 2$. Also, if $\kappa^2 \geqslant
3$, then Theorem~\ref{thm:BFTCylinder} implies that $\pm \n$ are the only two
global minimizers of $\mathcal{F}$, and these belong to the homotopy class
$\jdeg \m_{\bot} = 1$. Therefore, if $\kappa^2 \geqslant 3$, then any global
minimizer has degree one, and if $\kappa^2 < 2$, then any global minimizer has
degree zero.

Note that, at this stage, we cannot infer the existence of a threshold $\kappa_{\ast}^2
\in [2, 3]$ such that any global minimizer has degree one when $\kappa^2 >
\kappa_{\ast}^2$, and any global minimizer has degree zero when $\kappa^2 <
\kappa_{\ast}^2$. Indeed, in principle, it can still be the case that the minimizing homotopy classes are not described by intervals.
However, from the characterization of the minimizer in the homotopy classes
$\jdeg \m_{\bot} = 0$ (given in the proof of ii.) and from the comparison of
the associated energies (given in the proof of iii.), it immediately follows
that there exists a unique $\kappa^2_{\ast}$
\begin{equation}
  \kappa^2_{\ast} \in [2, 3] \label{eq:estkstar}
\end{equation}
such that if $\kappa^2 > \kappa_{\ast}^2$, then any global minimizer has
degree one, and if $\kappa^2 < \kappa_{\ast}^2$, then any global minimizer has
degree zero.

\medskip{\noindent}{\emph{Proof of i.}} It remains to show that for every $\kappa >
0$, the vector fields $\pm \n$ are the only two global minimizers of
$\mathcal{F}$ in the homotopy class $\jdeg \m_{\bot} = 1$. For that, it is
sufficient to observe that when $j = 0$, from \eqref{eq:estvaltemp2}, we get
that $\mathcal{F} (\m_{\bot}) \geqslant 2 \pi =\mathcal{F} (\pm \n)$
regardless of the value of $\kappa^2 > 0$. Moreover, $\mathcal{F} (\m_{\bot})
> 2 \pi$ if $\| \sin \theta \|^2_{L^2 [- \pi, \pi]} \neq 0$, i.e., if
$\m_{\bot} \notin \{\pm \n \}$. This guarantees the uniqueness statement about
the minimizers $\pm \n$.

\medskip{\noindent}{\emph{Proof of ii.}} We want to improve the estimate
\eqref{eq:estkstar}, but we also want to identify the minimizers of
$\mathcal{F}$ in the prescribed homotopy class $\jdeg \m_{\bot} = 0$. This
amounts to the minimization problem for $\mathcal{F}$ under the constraint $j
= - 1$ in \eqref{eq:defj}. In other words, one has to minimize energy
({\tmabbr{cf.}}~\eqref{eq:valtemp1})
  \begin{equation}
    \mathcal{F} (\m_{\bot}) = - 2 \pi + \int_{- \pi}^{\pi} |
    \partial_t \theta (t) |^2 + \kappa^2 \sin^2 \theta (t) \mathd t
    \label{eq:energyjm1}
  \end{equation}
  under the constraint that $\theta (\pi) - \theta (- \pi) = - 2 \pi$
  ({\tmabbr{cf.}}~\eqref{eq:defj}). The Euler--Lagrange equations associated
  with \eqref{eq:energyjm1} gives the equation
  \begin{equation}
    \partial_{t  t} \theta (t) = \kappa^2 \sin \theta (t) \cos \theta (t)
    \label{eq:ELequastionpend},
  \end{equation}
  subject to the degree constraint
  \begin{equation}
    \theta (\pi) = \theta (- \pi) - 2 \pi  . \label{eq:ELequastionpenddeg0}
  \end{equation}
  It is worth noticing that, from the mechanical point of view, the nonlinear
  ordinary differential equation \eqref{eq:ELequastionpend} describes the
  dynamics of an (ideal) inverted pendulum in the reduced setting where the
  pivot point of the pendulum is fixed in space
  ({\tmabbr{cf.}}~Figure~\ref{fig:revpend}). In this reduced setting, the
  equation of the inverted pendulum, up to a sign, is the one of a simple
  pendulum (see, e.g., {\cite[p.~35]{amann2011ordinary}}), the difference
  being in the nominal location of the pivot point which, for the inverted
  pendulum, is below its center of mass. Precisely, if the inverted pendulum
  consists of a spherical mass subject to the force of gravity $g$, placed at
  the end of a rigid massless rod of length $\ell$ then, its equation is given
  by \eqref{eq:ELequastionpend} with $\kappa^2 = g / \ell$. Using this
  mechanical analogy, the minimizers $\m_{\bot}$ in the class $\left\{ \jdeg
  \m_{\bot} = 0 \right\}$ we are interested in, correspond to
  solutions of the inverted pendulum in which the mass $m$ performs a full
  clockwise turn at the minimal cost of the energy \eqref{eq:energyjm1}.
\begin{figure}[t]
   {\includegraphics[width=\textwidth]{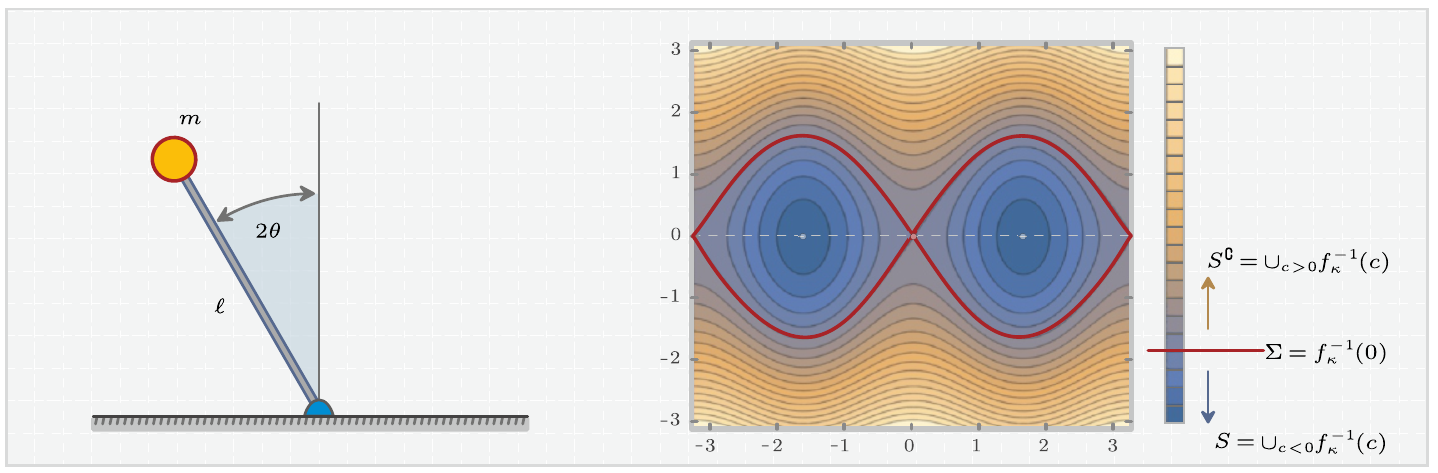}}
    \caption{\label{fig:revpend}(Left) The ideal inverted pendulum consists of
    a spherical mass $m$, subject to the force of gravity $g$, placed at the
    end of a rigid massless rod of length $\ell$ attached to a (possibly
    oscillating) pivot point. When the pivot point of the pendulum is fixed in
    space, the equation of motion, up to a sign, is the one of a simple
    pendulum; the difference is in the nominal location of the pivot point
    which, for the inverted pendulum, is below its center of mass. (Right) The
    typical phase portrait of the inverted pendulum
    \eqref{eq:ELequastionpend}. It consists of the level sets of the function
    $f_{\kappa} (x, y) \assign y^2 - \kappa^2 \sin^2 x$. The maximal quote of the separatrix
    $\Sigma:=f_\kappa^{-1}(0)$ is achieved at $|\kappa|$; here $\kappa^2 = 1.5$. The closed
    level curves correspond to oscillations of the pendulum about its
    equilibrium position $2 \theta = \pm \pi$, while the curves outside the
    separatrix $\Sigma$ correspond to full rotations of the pendulum.
    }
  \end{figure}  
 
  The problem is solvable in terms of elliptic integrals. For that, we observe
  that by multiplying both parts of \eqref{eq:ELequastionpenddeg0} by
  $\partial_t{\theta}$ one gets that if the integral curve
  \begin{equation}
    t \in [- \pi, \pi] \mapsto (\theta (t), \partial_t \theta (t)) \in \RR^2
  \end{equation}
  solves \eqref{eq:ELequastionpenddeg0} then there exists a constant
  $c_{\kappa} \in \RR$ such that
  \begin{equation}
    (\partial_t \theta (t))^2 - \kappa^2 \sin^2 \theta (t) = c_{\kappa} \quad
    \forall t \in [- \pi, \pi] . \label{eq:levelset}
  \end{equation}
  Precisely, $c_{\kappa} \assign | \partial_t \theta (- \pi) |^2 - \kappa^2
  \sin^2 \theta (- \pi)$. In other words, every solution $(\theta (t),
  \partial_t \theta (t))$ of the boundary value problem
  \eqref{eq:ELequastionpenddeg0} belongs to some level set of the function
  $f_{\kappa} (x, y) \assign y^2 - \kappa^2 \sin^2 x$, with the understanding
  that we formally set $y \assign \partial_t \theta$ and $x \assign \theta$.
  The phase diagram is depicted in Figure~\ref{fig:revpend} where the thicker
  line represents the separatrix
  \begin{equation}
    \Sigma \assign \left\{ (x, y) \in \RR^2 \of y^2 = \kappa^2 \sin^2 x
    \right\}
  \end{equation}
  which bounds the region
  \begin{equation}
    S \assign \left\{ (x, y) \in \RR^2 \of y^2 \leqslant \kappa^2 \sin^2 x
    \right\} .
  \end{equation}
  Note that the phase diagram is periodic (in the $x$-direction) of period
  $\pi$. However, it is convenient to consider the range $- \pi \leqslant x =
  \theta \leqslant \pi$ of length $2 \pi$ because we are interested in
  solutions such that $\theta (\pi) - \theta (- \pi) = - 2 \pi$. From the
  phase diagram represented in Figure~\ref{fig:revpend}, it is clear that the
  solutions we are interested in are such that the coordinate $y = \partial_t
  \theta$ never vanishes, as this is the only way to connect, via a phase
  curve, two points of the phase portrait whose $x$ coordinates are $2 \pi$
  away. The rigorous proof follows by observing that the region $S$ includes
  the level sets $\{ f_{\kappa}^{- 1} (c) \}_{c < 0}$ which consist of
  disjoint compact subsets of $S$, each one of them projecting on the $x$-axis
  on a set of diameter less than $\pi$. However, the solution curves $\alpha
  (t) \assign (\theta (t), \partial_t \theta (t))$ we are interested in, have
  to satisfy \eqref{eq:ELequastionpenddeg0} and, therefore, must have a trace
  in the phase space whose projection on the $x$-axis has diameter $2 \pi$. It
  follows that any solution $\alpha (t) \assign (\theta (t), \partial_t \theta
  (t))$ of \eqref{eq:ELequastionpend}-\eqref{eq:ELequastionpenddeg0} lies on
  the level set $f_{\kappa}^{- 1} (c_{\kappa})$ for some $c_{\kappa} \geqslant
  0$. After that, we observe that the solutions lying on the level set
  $f_{\kappa}^{- 1} (0)$ whose projection on the $x$-axis have diameter $2 \pi$ correspond to the normal vector fields $\pm
  \n$; hence, from now on, we focus on solutions in
  $S^{\complement} = \cup_{c > 0} f^{- 1} (c)$. Given the expression
  of $c_{\kappa}$, this implies that
  \begin{equation}
    | \partial_t \theta (- \pi) |^2 > \kappa^2 \sin^2 \theta (- \pi) .
    \label{eq:boundthetapatpi}
  \end{equation}
  First, we note that if $(x, y) \in S^{\complement}$ then $y \neq 0$.
  Therefore, the solutions of
  \eqref{eq:ELequastionpend}-\eqref{eq:ELequastionpenddeg0} are such that
  $\partial_t \theta (t) \neq 0$ for every $t \in [- \pi, \pi]$. Thus,
  $\theta$ is either decreasing or increasing. But given the degree condition
  \eqref{eq:ELequastionpenddeg0} the solutions we are looking for have to be
  decreasing. Therefore, from \eqref{eq:levelset}, we get that
  \begin{equation}
    \partial_t \theta (t) = - \sqrt{| \partial_t \theta (- \pi) |^2 -
    \kappa^2 \sin^2 \theta (- \pi) + \kappa^2 \sin^2 \theta (t)} .
  \end{equation}
  It is convenient to set
  \begin{equation}
    \alpha_{\kappa} \assign \sqrt{| \partial_t \theta (- \pi) |^2 - \kappa^2
    \sin^2 \theta (- \pi)}, \label{eq:newexpralkappa}
  \end{equation}
  because, as we are going to show, the value of $\alpha_{\kappa}$ just
  defined coincides with the value defined by \eqref{eq:uniqsolalpha}. Note
  that $\alpha_{\kappa} > 0$ due to \eqref{eq:boundthetapatpi}. The expression of $\partial_t \theta
  (t)$ can be rewritten under the form
  \begin{equation}
    \partial_t \theta (t) = - \alpha_{\kappa} \sqrt{1 + (\kappa^2 /
    \alpha^2_{\kappa}) \sin^2 \theta (t)} . \label{eq:exprthetaprime}
  \end{equation}
  Next, we observe that the elliptic integral of the first kind $F_{\kappa}$ (defined in \eqref{eq:EllipticErev})
  is increasing (invertible) and vanishes at $-\pi$. Therefore from \eqref{eq:exprthetaprime} we deduce that
  \begin{equation}
    F_{\kappa} (\theta (t)) - F_{\kappa} (\theta (- \pi)) = - (t + \pi)
    \alpha_\kappa .
  \end{equation}
  In particular, evaluating the previous relation at $t = \pi$ and taking into
  account \eqref{eq:ELequastionpenddeg0} we obtain
  \begin{equation}
    F_{\kappa} (\theta (\pi)) - F_{\kappa} (\theta (- \pi)) = - 2 \pi
    \alpha_\kappa . \label{eq:forthecharofalpha1}
  \end{equation}
  On the other hand, the integrand defining $F_{\kappa}$ is periodic of period
  $\pi$ and, therefore, taking also into account that $\theta (\pi) = \theta
  (- \pi) - 2 \pi$, we obtain
  \begin{align}
    F_{\kappa} (\theta (\pi)) - F_{\kappa} (\theta (- \pi)) =\; &
    \int_{\theta (- \pi)}^{\theta (- \pi) - 2 \pi} \frac{1}{\sqrt{1 +
    (\kappa^2 / \alpha^2_{\kappa}) \sin^2 x}} \mathd x \nonumber\\
    =\; & - \int_{- \pi}^{\pi} \frac{1}{\sqrt{1 + (\kappa^2 /
    \alpha^2_{\kappa}) \sin^2 x}} \mathd x \nonumber\\
    =\; & - 2 \pi \beta_{\kappa} ,  \label{eq:forthecharofalpha2}
  \end{align}
 where we set $\beta_{\kappa} \assign 1/(2 \pi) F_\kappa(\pi)$.
  From \eqref{eq:forthecharofalpha1} and \eqref{eq:forthecharofalpha2} it
  follows that if $(\theta (t), \partial_t \theta (t))$ is a solution of our
  problem \eqref{eq:ELequastionpend}-\eqref{eq:ELequastionpenddeg0} then
  necessarily $\alpha_{\kappa} = \beta_{\kappa}$. Therefore the value of
  $\alpha_{\kappa}$ can be characterized as the unique solution of the
  equation ({\tmabbr{cf.}}~\eqref{eq:uniqsolalpha})
  \begin{equation}
    \frac{1}{2 \pi} \int_{- \pi}^{\pi} \frac{1}{\sqrt{\alpha^2_{\kappa} +
    \kappa^2 \sin^2 x}} \mathd x = 1.
  \end{equation}
  Once computed $\alpha_\kappa$, we can characterize the solutions of
  \eqref{eq:ELequastionpend}-\eqref{eq:ELequastionpenddeg0} using
  \eqref{eq:forthecharofalpha1}, which gives the one-parameter family of
  functions $\theta (t) = F^{- 1}_{\kappa} \left( F_{\kappa} (\theta (- \pi))
  - (t + \pi) \alpha_{\kappa}\right)$, $\theta (- \pi) \in \RR$, which,
  by the way, is of the form
  \begin{equation}
    \theta (t) = \tmop{am}_{\kappa} (- \alpha_\kappa t + b_{\kappa}), \qquad
    b_{\kappa} \assign F_{\kappa} (\theta (- \pi)) - \pi \alpha_{\kappa} \in
    \RR . \label{eq:oneparfamsols}
  \end{equation}
  This proves \eqref{eq:minmdegzero} and gives a parameterization of the
  family of solutions in terms of the initial condition $\theta (- \pi)$, or
  in terms of the initial condition $\partial_t \theta (- \pi)$ due to
  \eqref{eq:newexpralkappa}.
  
  In principle, the energy can depend on $b_{\kappa}$, but this is not the
  case as we are going to show next. For that, we observe that from
  \eqref{eq:exprthetaprime} we get that
  \begin{equation}
    | \partial_t \theta (t) |^2 = \alpha_{\kappa}^2 + \kappa^2 \sin^2
    \theta (t) .
  \end{equation}
  Plugging the previous expression into the energy functional
  \eqref{eq:energyjm1}, we obtain that if $\m_{\bot}$ minimizes the energy in
  the homotopy class $\left\{ \jdeg \m_{\bot} = 0 \right\}$, then
  \begin{equation}
    \mathcal{F} (\m_{\bot}) = - 2 \pi (1 + \alpha_{\kappa}^2) + 2
    \int_{- \pi}^{\pi} | \partial_t \theta (t) |^2 \mathd t.
    \label{eq:jnewexpren}
  \end{equation}
  Next, we observe that since $\theta$ is a decreasing function, we have that
  \begin{align}
    \int_{- \pi}^{\pi} | \partial_t \theta (t) |^2 \mathd t =\; &
    \int_{\theta (\pi)}^{\theta (- \pi)}  | (\partial_t \theta  \circ
    \theta^{- 1}) (x) |^2  | \partial_x \theta^{- 1} (x) | \mathd x
    \nonumber\\
    =\; & \int_{\theta (\pi)}^{\theta (- \pi)}  \frac{1}{| \partial_x
    \theta^{- 1} (x) |} \mathd x. 
  \end{align}
  But from \eqref{eq:oneparfamsols} we know that $\theta^{- 1} (x) = \frac{b_\kappa
  - F_{\kappa} (x)}{\alpha_\kappa}$ and, therefore
  \begin{align}
    \int_{- \pi}^{\pi} | \partial_t \theta (t) |^2 \mathd t =\; &
    \alpha_{\kappa} \int_{\theta (\pi)}^{\theta (- \pi)}  \frac{1}{|
    F'_{\kappa} (x) |} \mathd x \nonumber\\
    =\; & \alpha_{\kappa} \int_{\theta (\pi)}^{\theta (- \pi)}  \sqrt{1 +
    (\kappa^2 / \alpha^2_{\kappa}) \sin^2 x} \, \mathd x. 
  \end{align}
  Making use of the $\pi$-periodicity of the integrand, from \eqref{eq:EllipticE} and \eqref{eq:ELequastionpenddeg0} we infer that
  \begin{equation}
    \int_{- \pi}^{\pi} | \partial_t \theta (t) |^2 \mathd t =
    \alpha_{\kappa} \int_{- \pi}^{\pi}  \sqrt{1 + (\kappa^2 /
    \alpha^2_{\kappa}) \sin^2 x} \, \mathd x =  \alpha_{\kappa} E_{\kappa}.
  \end{equation}
  Overall, plugging the previous expression into the expression
  \eqref{eq:jnewexpren} of the energy, we get
  \begin{equation}
    \mathcal{F} (\m_{\bot}) = - 2 \pi (1 + \alpha_{\kappa}^2) + 2
    \alpha_{\kappa} E_\kappa
    \label{eq:enlevdegreezerotemp}
  \end{equation}
  and this proves \eqref{eq:minendegzero}.
  
  {\noindent}{\emph{Proof of iii.}} The proof quickly follows from {\emph{i}}
  and {\emph{ii}} because the energy of the normal vector fields $\pm \n$
  evaluates to $2 \pi$ and, therefore, due to \eqref{eq:enlevdegreezerotemp},
  degree zero configurations given by \eqref{eq:oneparfamsols} are preferred
  as soon as
  \begin{equation}
    - 2 \pi (1 + \alpha_{\kappa}^2)+ 2
    \alpha_{\kappa} E_\kappa < 2 \pi .
  \end{equation}
  This happens when $\kappa^2 < \kappa^2_{\ast} \approx 2.31742$.
  
  {\noindent}{\emph{Proof of iv.}} Finally, we want to show that the degree
  zero solutions \eqref{eq:oneparfamsols}, which we know to be global
  minimizers when $\kappa^2 < \kappa^2_{\ast}$, retain local stability for
  every $\kappa^2 > 0$. For that, we compute the second
  variation of the energy \eqref{eq:enmovframVal} which, for any $\phi \in
  H^1_{\sharp} ( [- \pi, \pi], \RR)$, reads as
  ({\tmabbr{cf.}}~\eqref{eq:enmovframVal22})
  \begin{equation}
    \mathcal{F}'' (\theta) (\phi) = \int_{- \pi}^{\pi} | \partial_t \phi (t)
    |^2 + \kappa^2 \phi^2 (t) \cos 2 \theta (t) \hspace{0.17em} \mathd t.
    \label{eq:secvarintheta}
  \end{equation}
  From the Euler-Lagrange equations \eqref{eq:ELequastionpend} we get that
  \begin{equation}
    \partial_{t t t} \theta = \kappa^2  (\cos 2 \theta) \partial_t \theta
    \label{eq:ELder} .
  \end{equation}
  Moreover, since $| \partial_t \theta | \geqslant c'_{\theta} > 0$ on the
  compact set $[- \pi, \pi]$ for some $c_{\theta} > 0$, we can use the Hardy
  decomposition trick (see
  {\cite{Di_Fratta_2015,Ignat2015,Ignat2016a,Ignat2016,Ignat2020}}) and say
  that any $\phi \in H^1_{\sharp} ( [- \pi, \pi], \RR)$ can be
  written as $\phi = (\partial_t \theta) \psi$ for some $\psi \in H^1_{\sharp}
  ( [- \pi, \pi], \RR)$. Therefore
  \begin{equation}
    \mathcal{F}'' (\theta) (\phi) = \int_{- \pi}^{\pi} | \partial_t \theta |^2
    | \partial_t \psi |^2 + | \partial_{t t} \theta |^2  | \psi |^2 +
    \partial_{t t} \theta \partial_t \theta \partial_t \left| {\psi^2} 
    \right| + \kappa^2  | \psi |^2 | \partial_t \theta |^2 (\cos 2 \theta) 
    \hspace{0.17em} \mathd t.
  \end{equation}
  Integrating by parts the previous expression and making use of
  \eqref{eq:ELder} we obtain
  \begin{equation}
    \mathcal{F}'' (\theta) (\phi) \, = \, \int_{- \pi}^{\pi} | \partial_t
    \theta |^2  | \partial_t \psi |^2 \hspace{0.17em} \mathd t \, \geqslant \,
    (c'_{\theta})^2 \int_{- \pi}^{\pi} | \partial_t \psi |^2 \hspace{0.17em}
    \mathd t.
  \end{equation}
  Finally, plugging $\theta = 0$ into \eqref{eq:secvarintheta}, we get that
  $\pm \n$ are uniform locally stable critical points for every
  $\kappa^2 > 0$. Therefore, since the lifting \eqref{eq:liftSob} preserves small $H^1$ neighborhoods, we get that $\pm \n$ are local minimizers of the energy $\mathcal{F}$. This
  concludes the proof.
\end{proof}

\section{Acknowledgments}

{\tmabbr{G.~Di~F.}} acknowledges support from the Austrian Science Fund (FWF)
through the special research program {\emph{Taming complexity in partial
differential systems}} (Grant SFB F65) and the  project
\emph{Analysis and Modeling of Magnetic Skyrmions} (grant P-34609). G. Di F. also thanks TU Wien and MedUni
Wien for their support and hospitality.

The work of the {\tmabbr{V.~S.}} was
supported by the EPSRC grant EP/K02390X/1 and the Leverhulme grant
RPG-2018-438. {\tmabbr{G.~Di~F.}} and {\tmabbr{V.~S.}} would like to thank the
Max Planck Institute for Mathematics in the Sciences in Leipzig for support
and hospitality. {\tmabbr{G.~Di~F.}}, {\tmabbr{A.~F.}}, and {\tmabbr{V.~S.}}
acknowledge support from ESI, the Erwin Schr{\"o}dinger International
Institute for Mathematics and Physics in Wien, given in occasion of the
Workshop on {\emph{New Trends in the Variational Modeling and Simulation of
Liquid Crystals}} held at ESI, in Wien, on December 2-6, 2019.

{\noindent}{\tmname{Giovanni Di Fratta}}, Dipartimento di Matematica e Applicazioni “R. Caccioppoli”, Università degli Studi di Napoli “Federico II”, Via Cintia, Complesso Monte S. Angelo, 80126 Napoli, Italy. {\emph{E-mail address}}: {\tmsamp{giovanni.difratta@unina.it}}

{\noindent}{\tmname{Alberto Fiorenza}}, Dipartimento di Architettura,
Universit{\`a} di Napoli Federico II, Via Monteoliveto, 3, I-80134 Napoli,
Italy, and Istituto per le Applicazioni del Calcolo ``Mauro Picone'', sezione
di Napoli, Consiglio Nazionale delle Ricerche, via Pietro Castellino, 111,
I-80131 Napoli, Italy.
{\emph{E-mail address}}: {\tmsamp{fiorenza@unina.it}}

{\noindent}{\tmname{Valeriy Slastikov}}, School of Mathematics, University of
Bristol, University Walk, Bristol BS8 1TW, United Kingdom.
{\emph{E-mail address}}: {\tmsamp{valeriy.slastikov@bristol.ac.uk}}

\end{document}